\documentclass[12pt]{amsart}
\usepackage{cases}
\usepackage{amsthm}
\usepackage{amsmath}
\usepackage{amscd}
\usepackage{graphicx}
\usepackage[mathscr]{eucal}
\usepackage[colorlinks,linkcolor=blue,citecolor=blue, pdfstartview=FitH]{hyperref}

\input xy
\xyoption{all} \numberwithin{equation}{section}
\begin{document}

\title[Cyclotomic expansions for double twist knots]
{Cyclotomic expansions for double twist knots with an odd number of half-twists}
\author[Qingtao Chen, Kefeng Liu and Shengmao Zhu]{Qingtao Chen, Kefeng Liu and 
Shengmao Zhu}

\address{Division of Science \\
New York University Abu Dhabi \\
Abu Dhabi \\
United Arab Emirates} \email{chenqtao@nyu.edu, chenqtao@hotmail.com}
\address{Mathematical Science Research Center \\
Chongqing University of Technology  \\
Chongqing, P. R. China. }
\address{Department of mathematics \\
University of California at Los Angeles, Box 951555\\
Los Angeles, CA, 90095-1555.} \email{liu@math.ucla.edu}

\address{Department of Mathematics \\
Zhejiang Normal University  \\
Jinhua Zhejiang,  321004, China }
\email{szhu@zju.edu.cn}

\begin{abstract}
In this note, we compute the cyclotomic expansion formula for colored Jones polynomial of double twist knots with an odd number of half-twists $\mathcal{K}_{p,\frac{s}{2}}$ by using the Kauffman bracket skein theory. It answers a question proposed by Lovejoy and Osburn in 2019.   
\end{abstract}

\maketitle

\theoremstyle{plain} \newtheorem{thm}{Theorem}[section] \newtheorem{theorem}[%
thm]{Theorem} \newtheorem{lemma}[thm]{Lemma} \newtheorem{corollary}[thm]{%
Corollary} \newtheorem{proposition}[thm]{Proposition} \newtheorem{conjecture}%
[thm]{Conjecture} \theoremstyle{definition}
\newtheorem{remark}[thm]{Remark}
\newtheorem{remarks}[thm]{Remarks} \newtheorem{definition}[thm]{Definition}
\newtheorem{example}[thm]{Example}





\tableofcontents
\newpage

\section{Introduction}
In this paper, we consider the double twist knot $\mathcal{K}_{p,\frac{s}{2}}$ with an odd number of half-twists, where $p$ is a nonzero integer  and $s$ is odd. For convenience, we usually write $s=2m-1$ with $m$ a positive integer.

\begin{figure}[!htb] 
\begin{align}
\raisebox{-15pt}{
\includegraphics[width=120 pt]{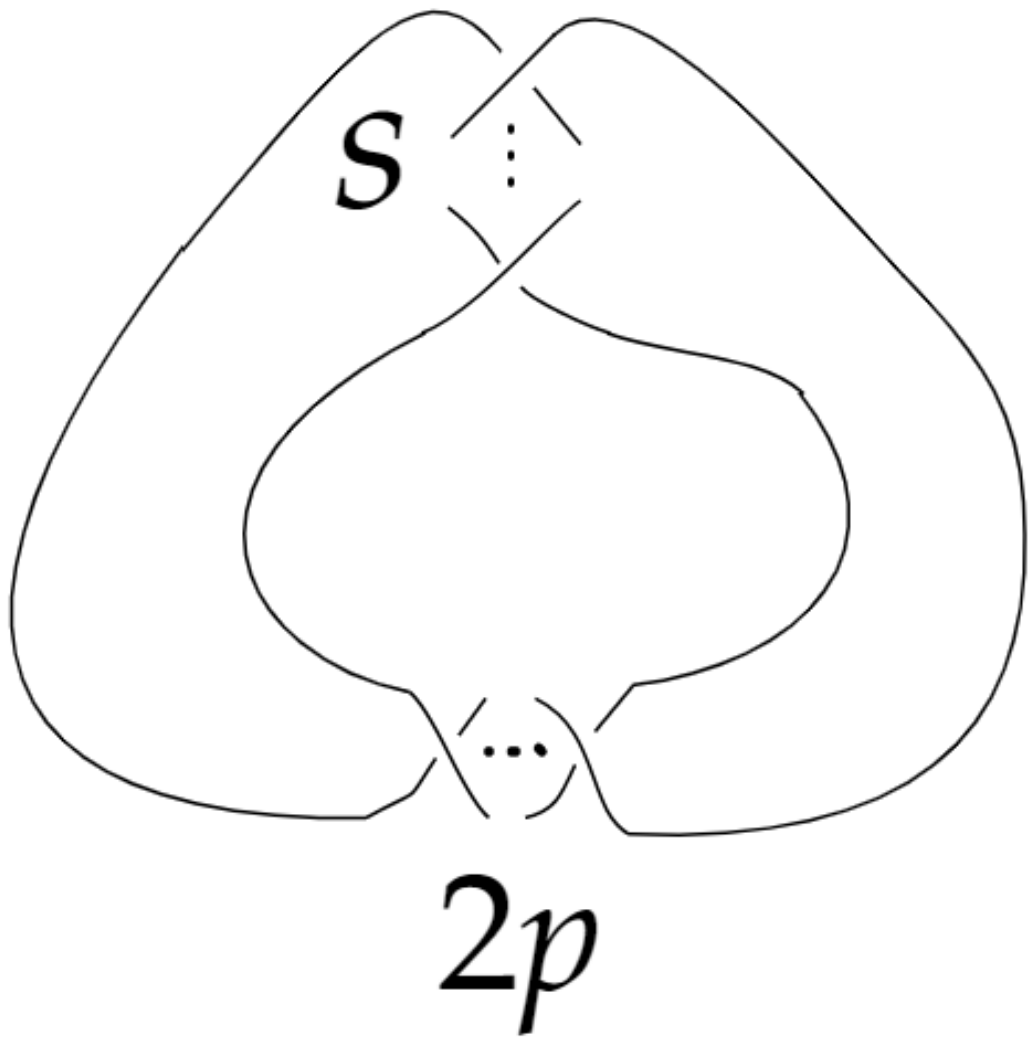}}
\end{align}
\caption{Double twist knot $\mathcal{K}_{p,\frac{s}{2}}$}
\end{figure}

 The goal of this note is to compute the cyclotomic expansion formula for the colored Jones polynomial of $\mathcal{K}_{p,\frac{s}{2}}$. Hence it answers a question proposed in \cite{LO19}. 

First, let us fix the notations used in this paper. Let $A$ be an indeterminate and $\mathfrak{q}=A^2$. For an integer $n$, we define the symbols by
\begin{align*}
[n] &= \frac{\mathfrak{q}^n-\mathfrak{q}^{-n}}{\mathfrak{q}-\mathfrak{q}^{-1}}, \qquad \{ n \} =  \mathfrak{q}^{n} -
\mathfrak{q}^{-n}.
\end{align*}
 Furthermore, for $n\geq 0$ and $0\leq i\leq n$,  we define the $\mathfrak{q}$-factorials by
\begin{align*}
[n]! =[n][n-1]\cdots [1], \quad \{ n \} ! =  \{ n \}\{n-1\}\cdots \{1\}, \quad \left[
\begin{array}{@{\,}c@{\,}}n \\ i \end{array} \right] =
\frac{[n]!}{[i]! [n-i]!}.
\end{align*}
We introduce the variable $q=\mathfrak{q}^2$ and the $q$-Pochhammer symbol as follows
\begin{align}
    (x;q)_{k}=\prod_{j=0}^k(1-xq^j). 
\end{align}
Then we define 
\begin{align}
    \left[
\begin{array}{@{\,}c@{\,}}n \\ i \end{array} \right]_q =
\frac{(q;q)_n}{(q;q)_i(q;q)_{n-i}}.
\end{align}

Given a knot $\mathcal{K}$, we denote by $J'_{N}(\mathcal{K};\mathfrak{q})$ the $N$-th normalized colored Jones polynomial of $\mathcal{K}$ as in \cite{Mas03}.  
Habiro proved the famous cyclotomic expansion formula for $J'_{N}(\mathcal{K};\mathfrak{q})$ (see formula (6.5) in \cite{Hab08}) which states that
there are Laurent polynomials (independent of $N$)
$H_k(\mathcal{K};\mathfrak{q})\in \mathbb{Z}[\mathfrak{q}^{\pm 1}]$,  such that
\begin{align}
J_{N}'(\mathcal{K};\mathfrak{q})=\sum_{k=0}^{N-1}H_k(\mathcal{K};\mathfrak{q})\frac{\{N+k\}!}{\{N-1-k\}!\{N\}}.
\end{align}
Habiro proved the existence of such  polynomial $H_k(\mathcal{K};\mathfrak{q})$ via quantum group theory \cite{Hab08}. 
However, it is not easy to compute the explicit formula of $H_k(\mathcal{K};\mathfrak{q})$ for a given knot $\mathcal{K}$. For the twist knot $\mathcal{K}_p$ ( i.e.  $\mathcal{K}_{p,1}$ in our notation), the explicit formula for $H_k(\mathcal{K}_p;\mathfrak{q})$ was first computed by Habiro in \cite{Hab00} (see also \cite{Hab08}), and later given by Masbaum in \cite{Mas03} by using the Kauffman bracket skein theory. 

As to the double twist knot 
$\mathcal{K}_{p,r}$ (where $p,r$ are two nonzero integers), one can derive the following expression for $H_{k}(\mathcal{K}_{p,r};\mathfrak{q})$ with a slight modification of Masbaum's computations in \cite{Mas03} for the twist knot $\mathcal{K}_p$( cf. \cite{Lau10,Wal14,LO17} etc.)
\begin{align}
    H_{k}(\mathcal{K}_{p,r};\mathfrak{q})=(-1)^kc'_{k,p}c'_{k,r},
\end{align}
where 
\begin{align}
    c'_{k,p}=\{k\}!\sum_{l=0}^k\frac{(-1)^l\mathfrak{q}^{2pl(l+1)}\{2l+1\}}{\{k+l+1\}!\{k-l\}!}.
\end{align}
Moreover, Masbaum also provided another multiple sum expression for the formula $c'_{k,p}$ (see formula (46) in \cite{Mas03}) which implies that 
$
  H_{k}(\mathcal{K}_{p,r};\mathfrak{q})\in \mathbb{Z}[\mathfrak{q}^{\pm 1}]. 
$

Now we consider the double twist knots with an odd number of half-twists $\mathcal{K}_{p,\frac{s}{2}}$.
 In this note, we compute an explicit formula for $H_k(\mathcal{K}_{p,\frac{s}{2}};\mathfrak{q})$ by using the Kauffman bracket skein theory.  The main result is
\begin{theorem} \label{theorem-1}
The colored Jones polynomial of $\mathcal{K}_{p,\frac{s}{2}}$ is given by 
\begin{align} \label{formula-JN}
J'_{N}(\mathcal{K}_{p,\frac{s}{2}};\mathfrak{q})=\sum_{k=0}^{N-1}H_{k}(\mathcal{K}_{p,\frac{s}{2}};\mathfrak{q})\frac{\{N+k\}!}{\{N-1-k\}!\{N\}},
\end{align}
where
\begin{align} \label{formula-Hk0}
H_{k}(\mathcal{K}_{p,\frac{s}{2}};\mathfrak{q})=(-1)^k\sum_{j=0}^kd_{k,j,p}c'_{j,p}
\tilde{c}'_{j,\frac{s}{2}},
\end{align}
with
\begin{align}
d_{k,j,p}&=\sum_{i=j}^k(-1)^{i+j}\mathfrak{q}^{-2pi(i+2)}\frac{\{2i+2\}\{i+1+j\}!}{\{k+i+2\}!\{k-i\}!\{i-j\}!},\\\nonumber
c'_{j,p}&=\{j\}!\sum_{l=0}^j\frac{(-1)^l\mathfrak{q}^{2pl(l+1)}\{2l+1\}}{\{j+l+1\}!\{j-l\}!},\\\nonumber
\tilde{c}'_{j,\frac{s}{2}}&=\{j\}!\sum_{l=0}^j\frac{\mathfrak{q}^{sl(l+1)}\{2l+1\}}{\{j+l+1\}!\{j-l\}!}.
\end{align}
\end{theorem}

Furthermore, by using the technique of Bailley chains \cite{And84}, we can derive the following multiple sum expressions for the above formulas.
\begin{theorem} \label{theorem-2}
We have the following identities 
\begin{align} \label{formula-ckp}
    c'_{k,p}&=(-1)^kq^{\frac{k(k+3)}{4}}\sum_{k=k_|p|\geq \cdots \geq k_1\geq 0}\prod_{i=1}^{p-1}q^{k_i^{2}+k_i}\left[
\begin{array}{@{\,}c@{\,}}k_{i+1} \\ k_i \end{array} \right]_{q}  \ \ (\text{if $p>0$}), \\ \nonumber    
 c'_{k,p}&=q^{\frac{-k(k+3)}{4}}\sum_{k=k_p\geq \cdots \geq k_1\geq 0}\prod_{i=1}^{|p|-1}q^{-k_ik_{i+1}-k_i}\left[
\begin{array}{@{\,}c@{\,}}k_{i+1} \\ k_i \end{array} \right]_{q}  \ \ (\text{if $p<0$}),
\end{align}
\begin{align} \label{formula-cks2}
   \tilde{c}'_{k,\frac{s}{2}}=(-1)^kq^{\frac{k(k+3)}{4}}\sum_{k=k_m\geq \cdots\geq k_1\geq 0 }\frac{1}{(q;q)_{k_1}}\prod_{i=1}^{m-1}q^{k_i^2+k_i}\left[
\begin{array}{@{\,}c@{\,}}k_{i+1} \\ k_i \end{array} \right]_q,
\end{align}
\begin{align} \label{formula-dkjp0}
    d_{k,j,p}&=(-1)^{k-j}q^{(j+1)(j-k)-pj(j+2)}\frac{1}{(q;q)_{k-j}}\sum_{k-j=k_p\geq \cdots\geq k_1\geq 0 }\prod_{i=1}^{p-1}q^{-k_ik_{i+1}-k_i}\left[
\begin{array}{@{\,}c@{\,}}k_{i+1} \\ k_i \end{array} \right]_q \ \ (\text{if $p>0$}),\\ \nonumber 
d_{k,j,p}&=q^{\frac{k(k+3)-j(j+3)}{2}+|p|j(j+2)}\frac{1}{(q;q)_{k-j}}\sum_{k-j=k_{|p|}\geq \cdots\geq k_1\geq 0 }\prod_{i=1}^{|p|-1}q^{k_i^2+k_i}\left[
\begin{array}{@{\,}c@{\,}}k_{i+1} \\ k_i \end{array} \right]_q   \ \ (\text{if $p<0$}).  
\end{align}
\end{theorem}
\begin{remark}
Formula (\ref{formula-ckp}) was firstly proved by Habiro in \cite{Hab00,Hab08} and then rederived by Masbaum \cite{Mas03} by using the Kauffman bracket skein theory. The proof of formulas (\ref{formula-ckp}) and (\ref{formula-cks2}) via the technique of Bailey chains presented here is essentially given in \cite{LO19}. The formula (\ref{formula-dkjp0}) seems new.  
\end{remark}

From the above multiple sum expressions, one can see that 
\begin{align}
    c'_{k,p}\in \mathbb{Z}[\mathfrak{q}^{\pm 1}], \ \tilde{c}'_{k,\frac{s}{2}} \in \frac{1}{\{k\}!}\mathbb{Z}[\mathfrak{q}^{\pm 1}] \ \text{and} \ d_{k,j,p}\in \frac{1}{\{k-j\}!}\mathbb{Z}[\mathfrak{q}^{\pm 1}].
\end{align}
Although Habiro's cyclotomic expansion theorem tells us that $H_{k}(\mathcal{K}_{p,\frac{s}{2}};\mathfrak{q})$ lies in $\mathbb{Z}[\mathfrak{q}^{\pm 1}]$ and one can also verify this fact through numeric computations by using formula (\ref{formula-Hk0}),  it remains a challenge to prove that $H_{k}(\mathcal{K}_{p,\frac{s}{2}};\mathfrak{q})\in \mathbb{Z}[\mathfrak{q}^{\pm 1}]$  by  formula (\ref{formula-Hk0}) in Theorem \ref{theorem-1} directly.

\textbf{Acknowledgements.}
The authors would like to thank Robert Osburn for raising the question to us of how to compute the cyclotomic coefficients for the colored Jones polynomial of the double twist knots in the ``odd-even" case.

\section{Preliminaries}
In this section, we review the main results in \cite{Mas03}. Let $M$ be an oriented 3-manifold, the Kauffman bracket skeim module $\mathcal{K}(M)$ is the free $\mathbb{Z}[A^{\pm 1}]$-module generated by isotopy classes of banded links in $M$ modulo the submodule generated by  the Kauffman bracket skein relation.

By its definition, it is easy to see that the Kauffman bracket $\langle \ \rangle$ of the banded links gives that $\mathcal{K}(S^3)=\mathbb{Z}[A^{\pm 1}]$. We use the normalization that the bracket of the empty link is 1. 

Furthermore, the skein module of the solid torus $S^1\times D^2$ is given by $\mathbb{Z}[A^{\pm 1}][z]$. Usually, we denote this skein module by $\mathcal{B}$. Here $z$ is given by the banded link $S^1\times J$, where $J$ is a small arc lie in the interior of $D^2$, and $z^n$ means $n$-parallel copies of $z$. We define the even part $\mathcal{B}^{ev}$ of $\mathcal{B}$ to be the submodule generated by the even powers of $z$.     

Let $t: \mathcal{B}\rightarrow \mathcal{B}$ denote the twist map induced by a full right handed twist on the solid torus. Then there exists a basis $\{e_i\}_{i\geq 0}$ of eigenvectors for the twist map $t$ (see e.g \cite{BHMV92}), which is defined recursively by 
\begin{align}
    e_0=1, \ e_1=z,  \ e_i=ze_{i-1}-e_{i-2}. 
\end{align}
Moreover, the $e_i$ satisfies
\begin{align}
\langle e_i \rangle&=(-1)^i[i+1]  \\
t(e_i)&=\mu_i e_i
\end{align}
where $\mu_i=(-1)^iA^{i^2+2i}$ is also called the framing factor which provides the contribution of a framing adding to a banded link.  

\begin{definition}
Given a knot $\mathcal{K}$ with zero framing, the $N$-th colored Jones polynomial $J_{N}(\mathcal{K};\mathfrak{q})$ of $\mathcal{K}$ is defined to be the Kauffman bracket of $\mathcal{K}$ cabled by $(-1)^{N-1}e_{N-1}$, i.e.
\begin{align}
    J_{N}(\mathcal{K};\mathfrak{q})=(-1)^{N-1}\langle\mathcal{K}(e_{N-1}) \rangle
\end{align}
where the factor of $(-1)^N$ is included such that for the unknot $U$, $J_{N}(U;\mathfrak{q})=[N]$. Furthermore, the normalized $N$-th colored Jones polynomial of $\mathcal{K}$ is defined as
\begin{align}
    J'_{N}(\mathcal{K};\mathfrak{q})=\frac{\langle \mathcal{K}(e_{N-1})\rangle}{\langle e_{N-1}\rangle}.
\end{align}
\end{definition}

Let $\lambda_i=-\mathfrak{q}^{i+1}-\mathfrak{q}^{-i-1}$, we define the skein element $R_n\in\mathcal{B}$ as follow 
\begin{align}
R_n=\prod_{i=0}^{n-1}(z-\lambda_{2i}).    
\end{align}
Then $R_n$ forms a basis of $\mathcal{B}$. Therefore, there are coefficients $t_{k,i}$ and $s_{i,j}$ such that 
\begin{align} \label{formula-Rk}
    R_k=\sum_{i=0}^kt_{k,i}e_i 
\end{align}
and
\begin{align} \label{formula-ei}
e_i=\sum_{j=0}^is_{i,j}R_j.
\end{align}

By induction, we can check that 
\begin{align}
    t_{k,i}=\frac{\{2k+1\}!\{2i+2\}}{\{k+i+2\}!\{k-i\}!}
\end{align}
and 
\begin{align}
    s_{i,j}=(-1)^{i+j}\left[
\begin{array}{@{\,}c@{\,}}i+1+j \\ i-j \end{array} \right].
\end{align}

We also need the following two properties for $R_k$: 
\begin{align} \label{formula-Rke2i}
    \langle R_k,e_{2i}\rangle=0 \ \text{for} \ i<k,
\end{align}
and 
\begin{align}
    \langle R_k,e_{2k} \rangle=(-1)^k\frac{\{2k+1\}!}{\{1\}!}.
\end{align}

Habiro \cite{Hab08} introduced the following skein element $\omega_+=\sum_{k=0}^{\infty} c_{n,+}R_k\in \mathcal{B}$, where $c_{k,+}$ is given by  
\begin{align} \label{formula-Ck+}
    c_{k,+}=(-1)^k\frac{\mathfrak{q}^{k(k+3)/2}}{\{k\}!}.
\end{align}
A key property that $\omega_+$ satisfied is that 
\begin{align}
    \langle\omega_+,x \rangle=\langle t(x)\rangle
\end{align}
for any $x\in \mathcal{B}^{ev}$. Habiro proved this property by using the relationship with the quantum group $U_{q}(sl_2)$. In \cite{Mas03}, Masbaum gave another combinatoric proof by using the method of orthogonal polynomials developed in \cite{BHMV92}.

We write $\omega=\omega_+$ and put
\begin{align}
    \omega^p=\sum_{k=0}^{\infty}c_{k,p}R_k.
\end{align}
Since circling with $\omega^p$ is the same as circling with $p$ parallel copies of $\omega$, by formula (\ref{formula-Ck+}),  we have 
\begin{align}
    \langle\omega^p,x\rangle=\langle t^p(x)\rangle
\end{align}
for every even $x\in \mathcal{B}^{ev}$.  By using the Kauffman bracket theory \cite{MV94}, Masbaum computed the coefficient $c_{k,p}$ which is given as follows (see formula (33) in \cite{Mas03})

\begin{align}
    c_{k,p}=\sum_{l=0}^k\frac{(-1)^l\mathfrak{q}^{2pl(l+1)}\{2l+1\}}{\{k+l+1\}!\{k-l\}!}.
\end{align}
In particular,  we have
\begin{align}
    c_{k,1}=\frac{(-1)^k\mathfrak{q}^{\frac{k(k+3)}{2}}}{\{k\}!}.
\end{align}

For simplicity, we introduce the notation 
\begin{align}
  c'_{k,p}=\{k\}!c_{k,p}.   
\end{align}
By using the method used in \cite{Mas03},  we can derive the following formula for the colored Jones polynomial of the double twist knot $\mathcal{K}_{p,r}$ (see formula (6.7) in \cite{Lau10})   
\begin{align} \label{formula-JKps}
    J'_N(\mathcal{K}_{p,r};\mathfrak{q})=\sum_{k=0}^{N-1}(-1)^kc'_{k,p}c'_{k,r}\frac{\{N+k\}!}{\{N-1-k\}!\{N\}}. 
\end{align}

In order to prove the cyclotomic expansion formula for twist knot, Masbaum also gave another expression for the formula $c'_{k,p}$ (see formula (46) in \cite{Mas03}). The main advantage of this expression for $c'_{k,p}$ is that one can see that
\begin{align} \label{formula-ckp-integrality}
  c'_{k,p}\in \mathbb{Z}[\mathfrak{q}^{\pm 1}]
\end{align}
directly. 

\begin{remark}
By using the $q$-Pochhammer symbol,  we obtain
\begin{align} \label{formula-kckp}
c'_{k,p}=(-1)^kq^{\frac{k^2+3k}{4}}\sum_{l=0}^{k}(-1)^lq^{l(l+1)p+\frac{l(l-1)}{2}}\frac{(1-q^{2l+1})(q;q)_k}{(q;q)_{k+l+1}(q;q)_{k-l}},
\end{align}  
and
\begin{align}
\frac{\{N+k\}!}{\{N-1-k\}!\{N\}!}=(-1)^kq^{-\frac{k(k+1)}{2}}(q^{1-N};q)_k(q^{1+N};q)_k.    
\end{align}
Then the formula (\ref{formula-JKps}) can be rearranged as follows
\begin{align}
    J'_N(\mathcal{K}_{p,r};q)&=\sum_{k=0}^{N-1}\left(\sum_{l=0}^{k}(-1)^lq^{l(l+1)p+\frac{l(l-1)}{2}}\frac{(1-q^{2l+1})(q;q)_k}{(q;q)_{k+l+1}(q;q)_{k-l}}\right)\\\nonumber
    &\cdot\left(\sum_{l=0}^{k}(-1)^lq^{l(l+1)r+\frac{l(l-1)}{2}}\frac{(1-q^{2l+1})(q;q)_k}{(q;q)_{k+l+1}(q;q)_{k-l}}\right)q^k(q^{1-N};q)_k(q^{1+N};q)_k.
\end{align}
We will show in Section \ref{Section-Bailey}  that one can also prove that formula (\ref{formula-kckp}) lies in $\mathbb{Z}[\mathfrak{q}^{\pm 1}]$ by using the method of Bailley chain.   
\end{remark}

In order to prove our main theorem, we also need to recall some basic facts about the Kauffman bracket of addmissibly colored banded trivalent graphs as in \cite{Mas03,MV94}. A color is just an integer $\geq 0$. A triple of colors $(a,b,c)$ is admissible if $a+b+c\equiv 0 \ (\mod 2) $ and $|a-b|\leq c\leq |a+b|$.  Let $D$ be a planar diagram of a banded trivalent graph. An admissible coloring of $D$ is an assignment of colors meeting there form an admissible triple. The Kauffman bracket of $D$ is defined to be the bracket of the expansion of $D$ obtained as follows. The expansion of an edge colored $n$ consists of $n$ parallel strands with a copy of the $n$-th Jones-Wenzl idempotent inserted. The expansion of a vertex is defined as in  Figure 2, where the internal colors $i,j,k$ are defined by 
\begin{align}
    i=(b+c-a)/2,  \ j=(c+a-b)/2, \ k=(a+b-c)/2.
\end{align}

\begin{figure}[!htb] 
\begin{align}
\raisebox{-15pt}{
\includegraphics[width=60 pt]{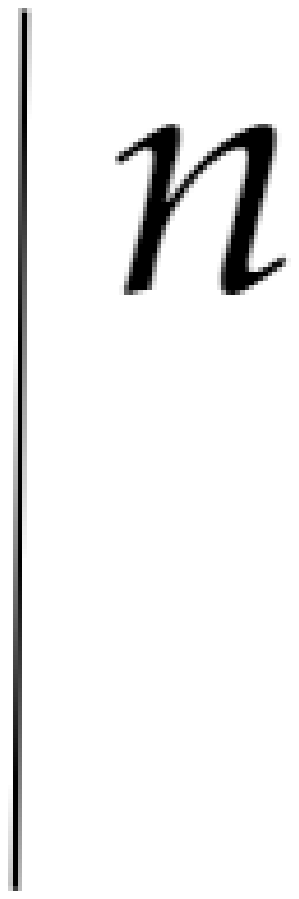}}=\raisebox{-15pt}{\includegraphics[width=60 pt]{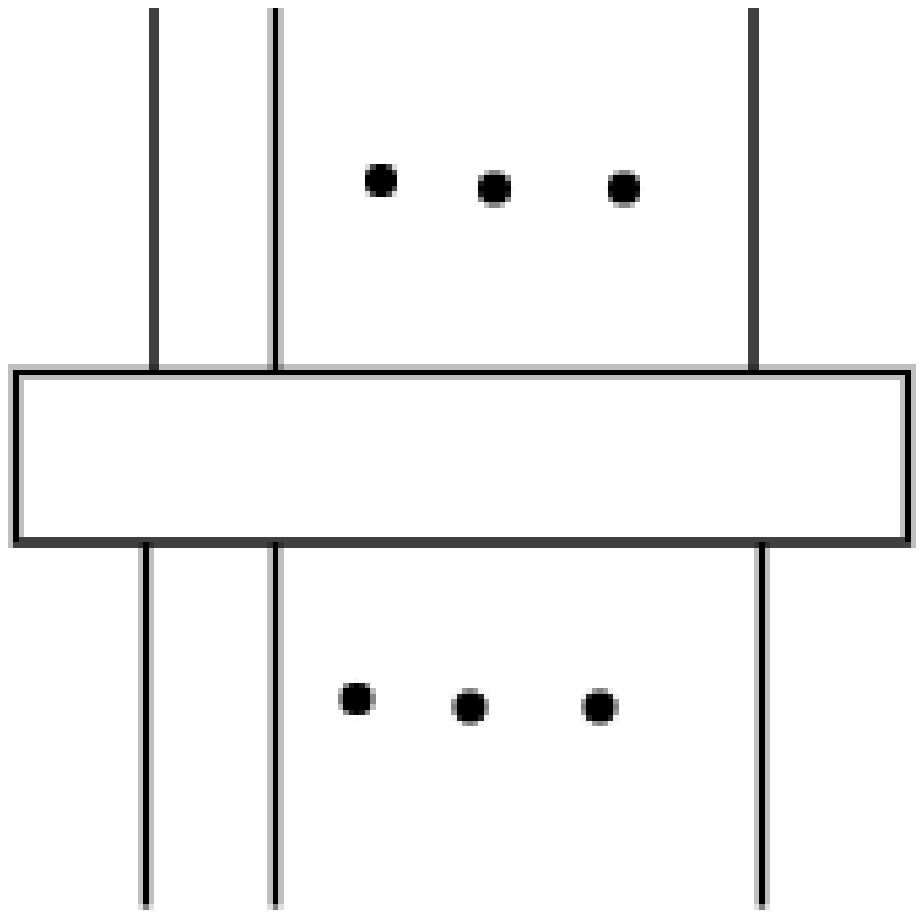}} 
\raisebox{-15pt}{
\includegraphics[width=60 pt]{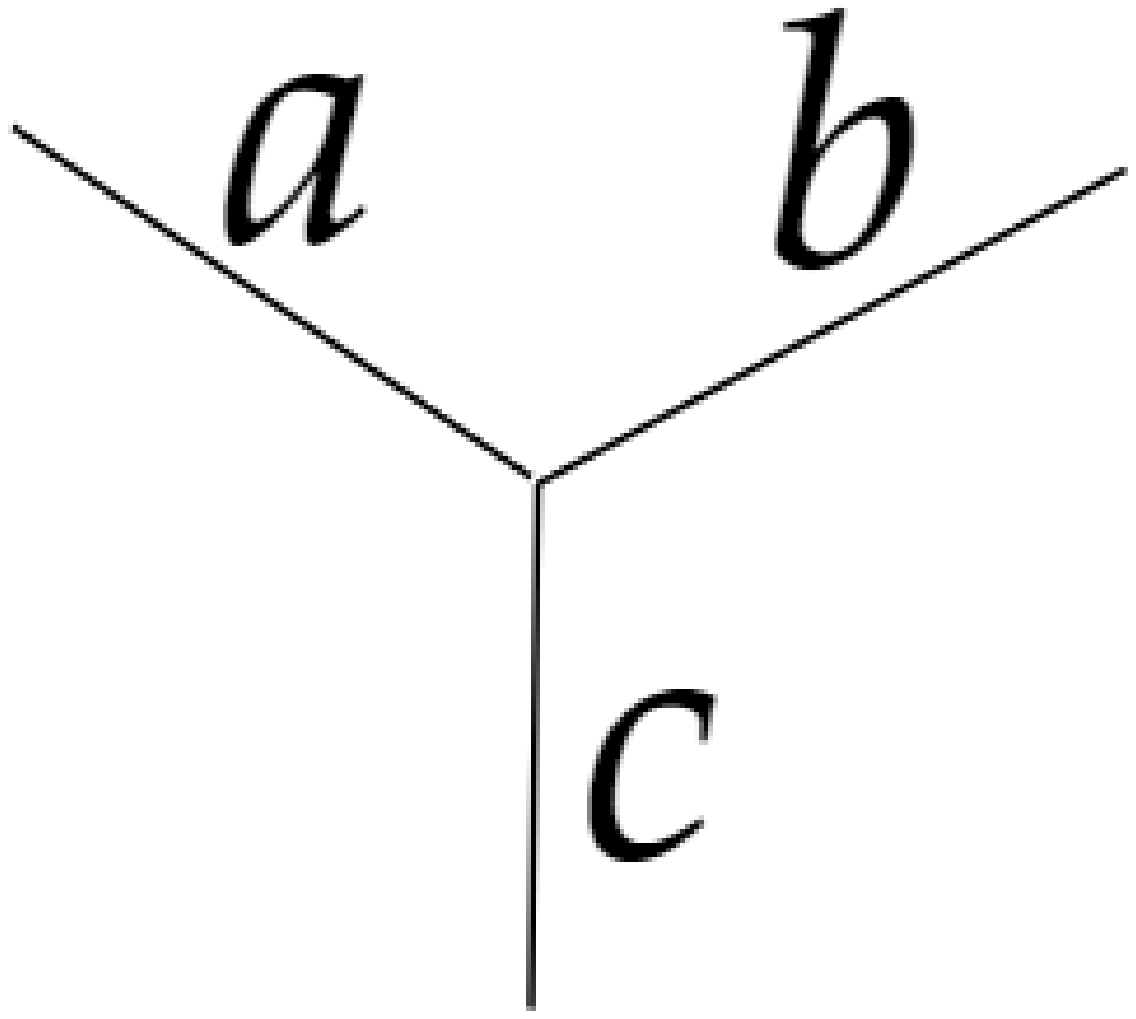}}=\raisebox{-15pt}{\includegraphics[width=60 pt]{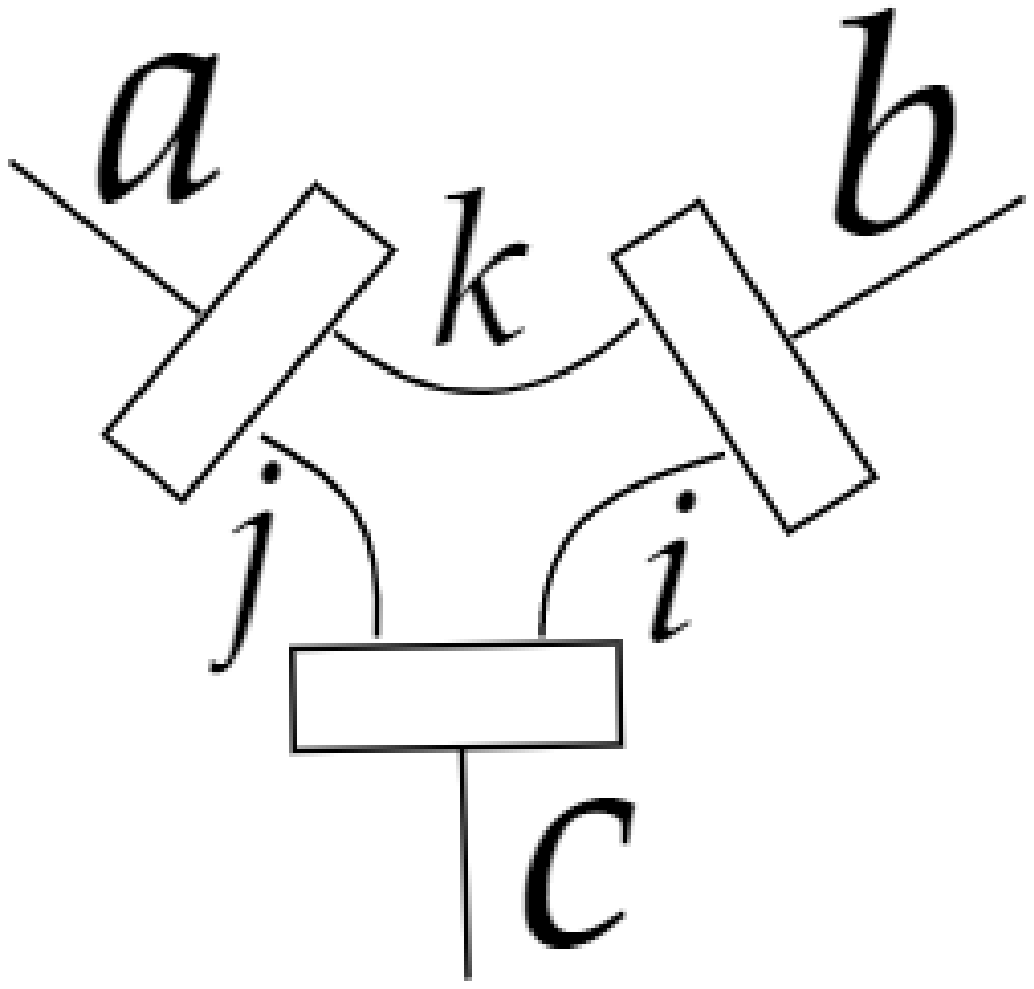}} 
\end{align}
\caption{n-th Jones-Wenzl idempotent and the expansion of a vertex}
\end{figure}

We have the fusion equation 

\begin{figure}[!htb]  
\begin{align} \label{formula-ab}
\raisebox{-25pt}{
\includegraphics[width=70 pt]{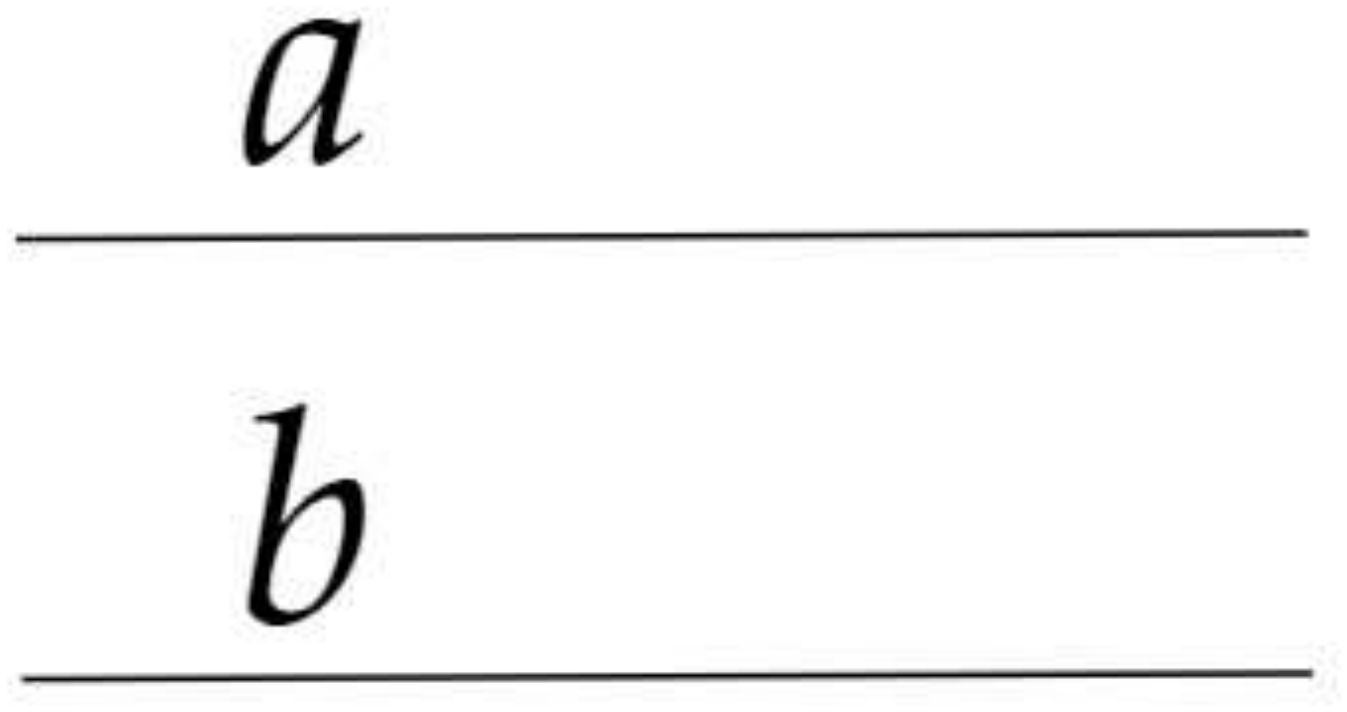}}&=\sum_{c}\frac{\langle c\rangle}{\langle a,b,c\rangle}\raisebox{-30pt}{
\includegraphics[width=70 pt]{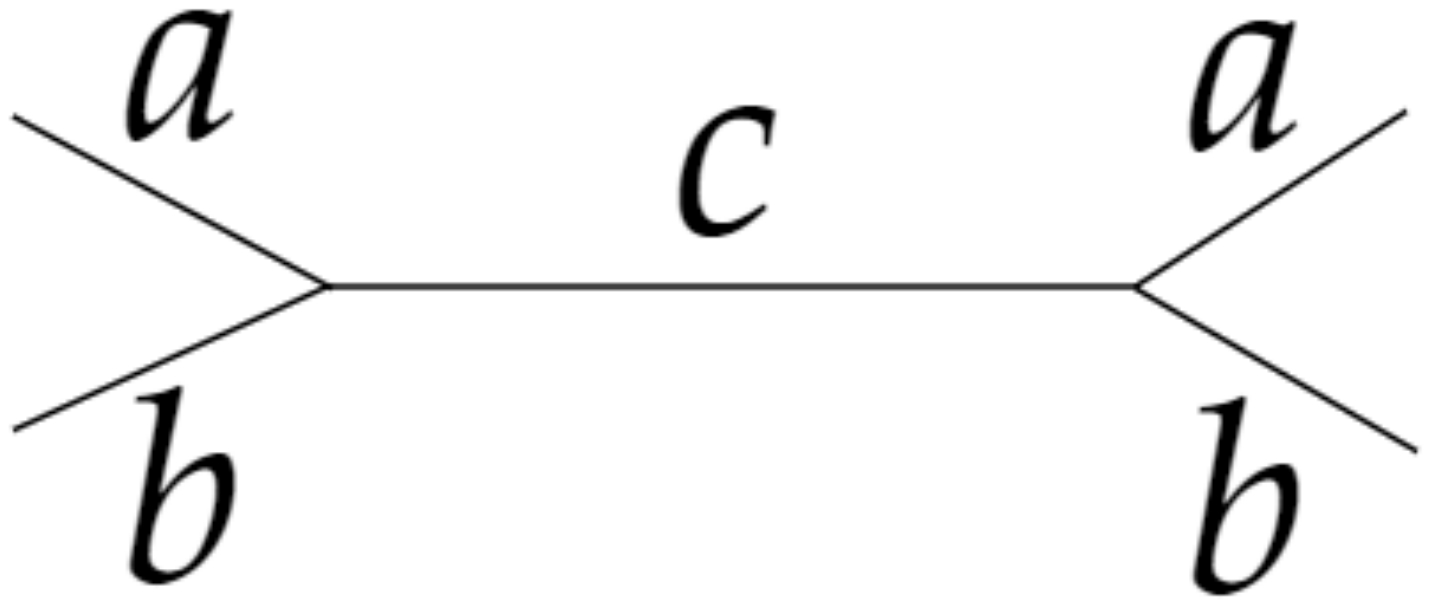}}
\end{align}
\caption{Fusion equation}
\end{figure}
where the sum is over the colors $c$ such that the triple $(a,b,c)$ is admissible, and we have 
\begin{align}
    \langle c \rangle=\langle e_c\rangle=(-1)^c[c+1]
\end{align}
and
\begin{align}
\langle a,b,c\rangle=\sum_{c}\frac{\langle c\rangle}{\langle a,b,c\rangle}\raisebox{-15pt}{
\includegraphics[width=60 pt]{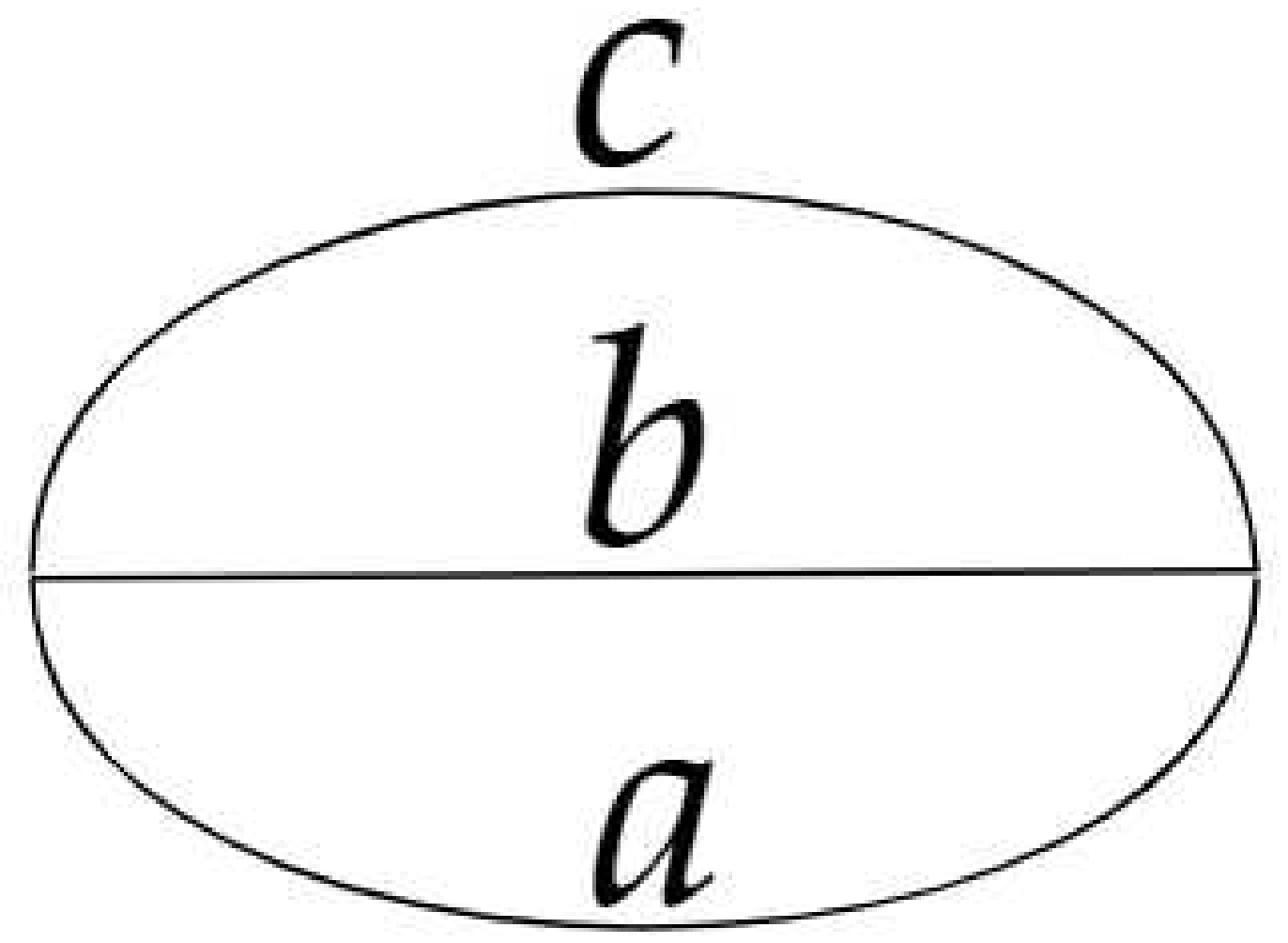}}=(-1)^{i+j+k}\frac{[i+j+k+1]![i]![j]![k]!}{[a]![b]![c]!}.
\end{align}
In particular, $\langle n,n,2n \rangle=\langle 2n \rangle$, and hence we have
\begin{align}
\raisebox{-25pt}{
\includegraphics[width=70 pt]{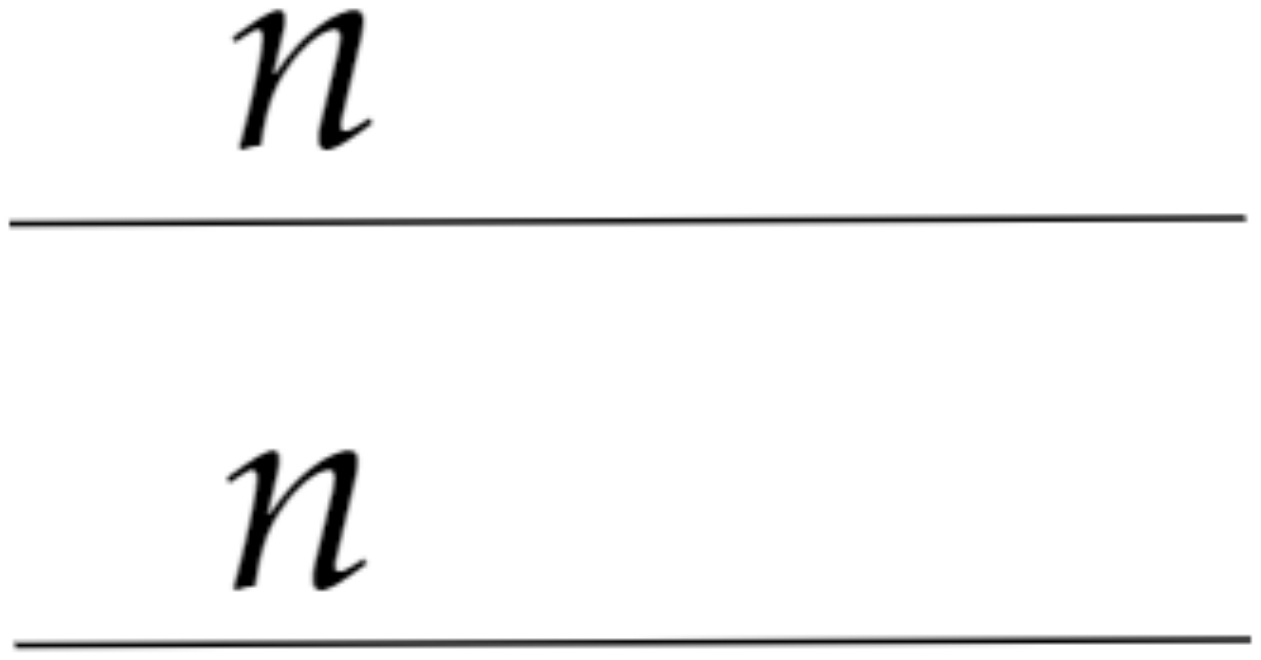}}=\raisebox{-35pt}{
\includegraphics[width=80 pt]{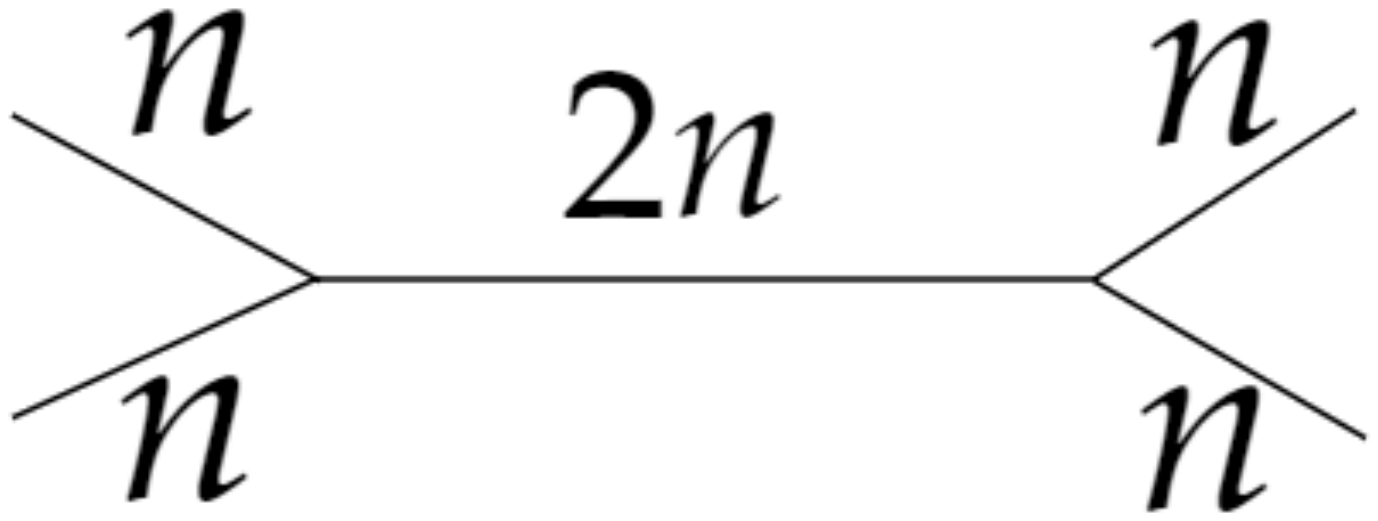}}+\cdots.
\end{align}

As an easy consequence of  the formula for the tetrahedron coefficient given in \cite{MV94}, we have
\begin{align} \label{formula-tetra}
\raisebox{-40pt}{
\includegraphics[width=120 pt]{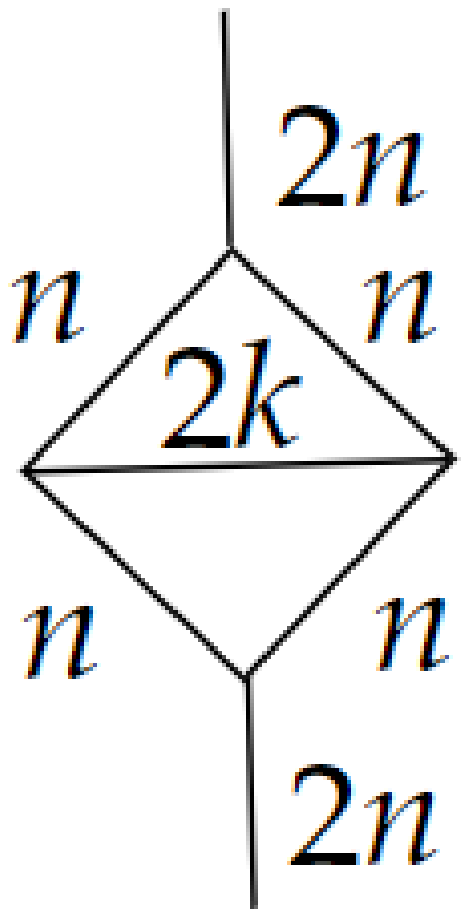}}=\frac{([k]!)^2}{[2k]!}\raisebox{-35pt}{
\includegraphics[width=120 pt]{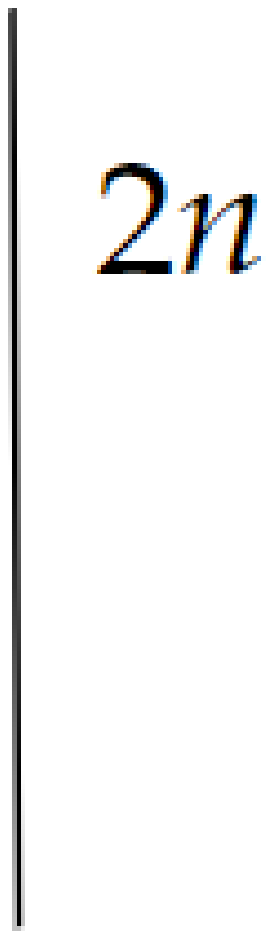}},
\end{align}
for $0\leq k\leq n$. 

Finally, we need the following formula (cf. Theorem 3 in \cite{MV94})
\begin{align} \label{formula-twist}
\raisebox{-30pt}{
\includegraphics[width=80 pt]{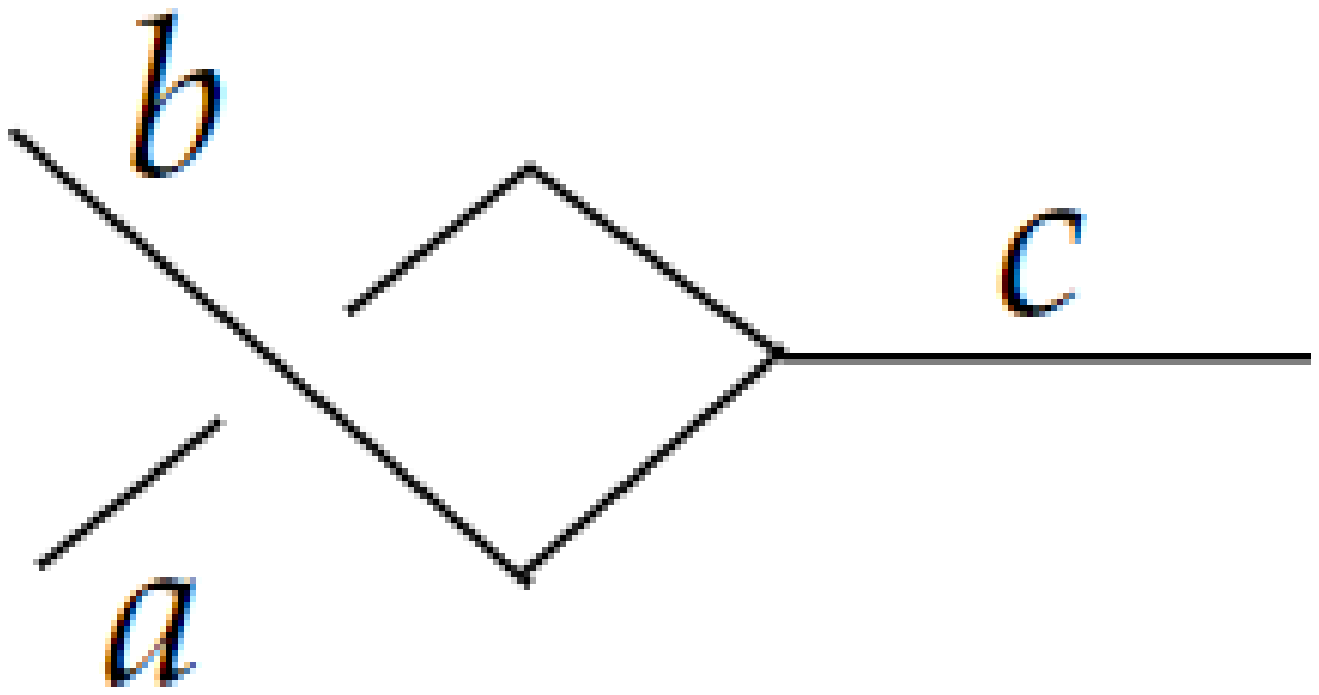}}=\delta(c;a,b)\raisebox{-20pt}{
\includegraphics[width=60 pt]{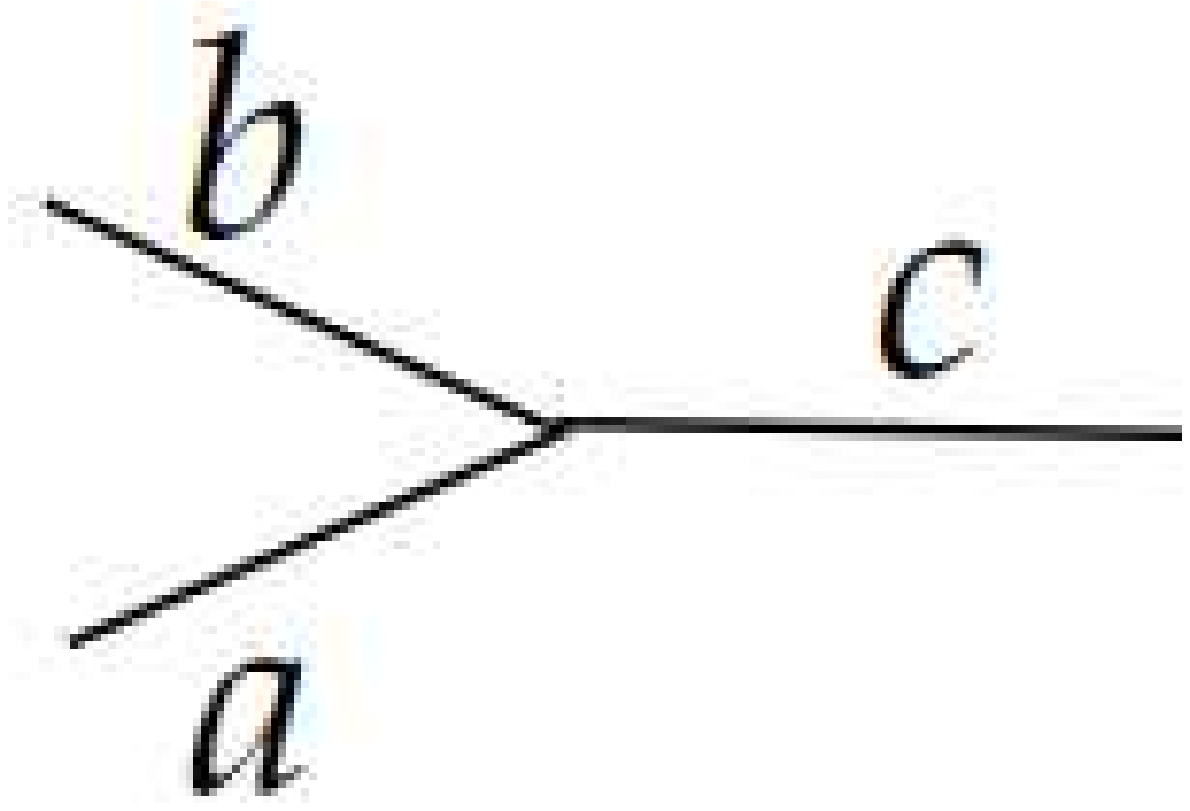}}
\end{align}
where the coefficient $\delta(c;a,b)$ is given by 
\begin{align} \label{formula-delta}
    \delta(c;a,b)=(-1)^{\frac{a+b-c}{2}}A^{-a-b+c-\frac{a^2+b^2-c^2}{2}}
\end{align}
and 
\begin{align} \label{formula-delta2}
    \delta(c;a,b)^2=\frac{\mu_c}{\mu_a\mu_b}.
\end{align}

\section{Proof of the cyclotomic expansion formula}
We write
 $t^p(R_k)=\sum_{j=0}^kd_{k,j}^{(p)}R_j$, then we have 
 \begin{lemma} \label{lemma-djk}
 The coefficient $d_{k,j}^{(p)}$ is given by
 \begin{align}
     d_{k,j}^{(p)}=\sum_{i=j}^k(-1)^{i(1+p)+j}\frac{\{2k+1\}!\{2i+2\}!}{\{k+i+2\}!\{k-i\}!}\mathfrak{q}^{p(i+2)i/2}\frac{\{i+1+j\}!}{\{i-j\}!\{2j+1\}!}.
 \end{align}
 \end{lemma}
\begin{proof}
This formula is a straightforward consequence of the formulas (\ref{formula-Rk}) and  (\ref{formula-ei}). Indeed, 
\begin{align}
    t^p(R_k)&=t^p(\sum_{i=0}^k t_{k,i}e_i) =\sum_{i=0}^kt_{k,i}t^p( e_i) =\sum_{i=0}^kt_{k,i}\mu_i^{p}e_i
    =\sum_{i=0}^k\mu_i^{p}\sum_{j=0}^is_{i,j}R_j\\\nonumber
    &=\sum_{j=0}^k\sum_{i=j}^kt_{k,i}\mu_i^ps_{i,j} R_j. 
\end{align}
So we have 
\begin{align}
    d_{k,j}^{(p)}=\sum_{i=j}^kt_{k,i}\mu_i^ps_{i,j}
\end{align}
\end{proof}

\begin{lemma} \label{Lemma-2n2n}
We have the following formula 
\begin{align}
\raisebox{-40pt}{
\includegraphics[width=110 pt]{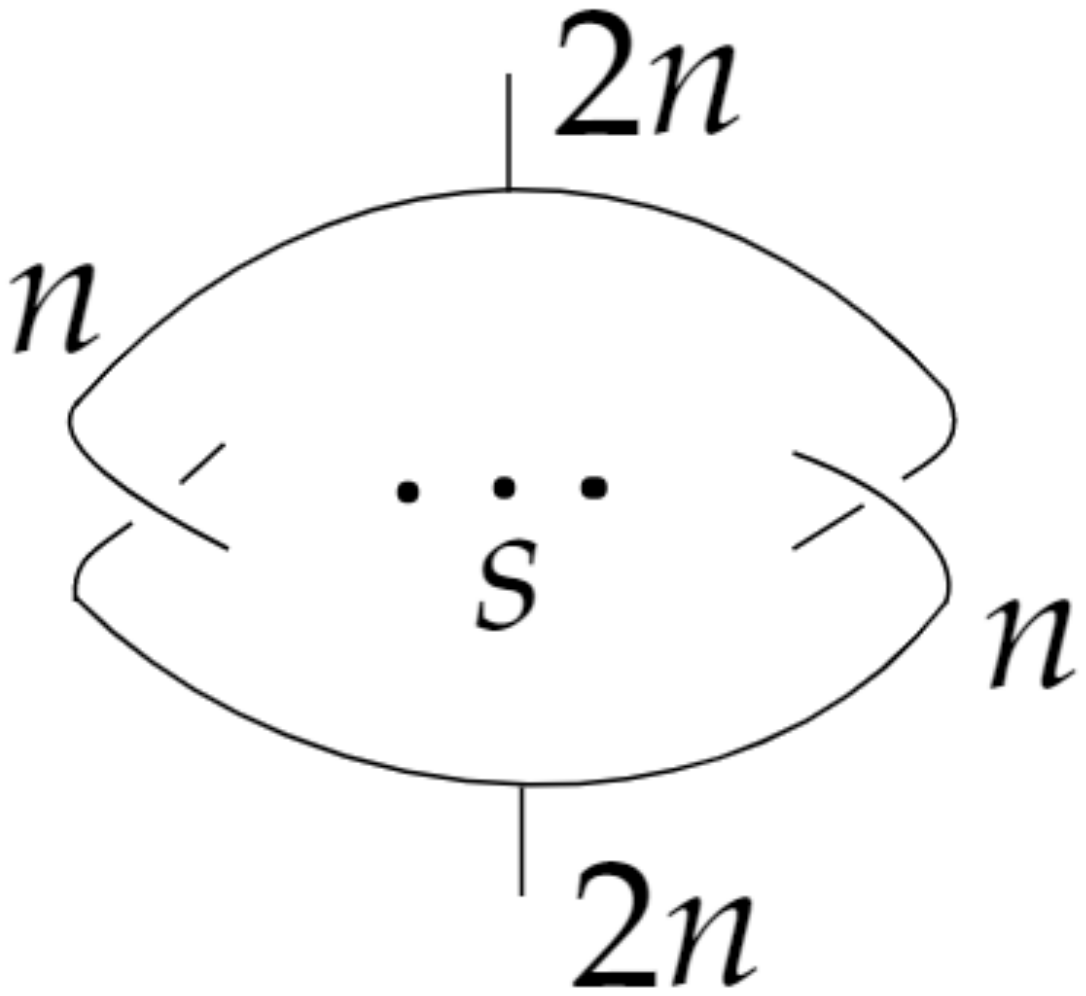}}&=\sum_{k=0}^n\frac{(-1)^n\mu_{2k}^{\frac{s}{2}}}{\mu_n^s}\frac{[2k+1]([n]!)^2}{[n+k+1]![n-k]!}\raisebox{-40pt}{
\includegraphics[width=110 pt]{2n.pdf}}.
\end{align}
\end{lemma}
\begin{proof}
By using formulas (\ref{formula-ab}), (\ref{formula-twist}) and (\ref{formula-tetra}), we have 
\begin{align}
\raisebox{-40pt}{
\includegraphics[width=110 pt]{2n2n.pdf}}&=\sum_{k=0}^{n}\delta(2k,n,n)^{s}\frac{\langle 2k \rangle}{\langle n,n,2k\rangle}\raisebox{-45pt}{
\includegraphics[width=130 pt]{2n2nnnnn.pdf}}\\\nonumber&=\sum_{k=0}^{n}\delta(2k,n,n)^{s}\frac{\langle 2k \rangle}{\langle n,n,2k\rangle}\frac{([k]!)^2}{[2k]!}\raisebox{-35pt}{
\includegraphics[width=120 pt]{2n.pdf}}.
\end{align}

By formula (\ref{formula-delta}), we obtain
\begin{align}
    \delta(2k;n,n)=(-1)^{n-k}\mathfrak{q}^{k^2+k}q^{-\frac{n^2}{2}-n}=\frac{(-1)^k(\mathfrak{q}^{2k^2+2k})^{\frac{1}{2}}}{(-1)^n
    \mathfrak{q}^{\frac{n^2+2n}{2}}}=\frac{(-1)^k\mu_{2k}^{\frac{1}{2}}}{\mu_n}.
\end{align}
Since $s-1$ is even, by applying formula (\ref{formula-delta2}), we obtain
\begin{align}
    \delta(2k;n,n)^s=\delta(2k;n,n)^{s-1}\delta(2k;n,n)=
    \frac{\mu_{2k}^{\frac{s-1}{2}}}{\mu_n^{s-1}}\frac{(-1)^k\mu_{2k}^{\frac{1}{2}}}{\mu_n}=\frac{(-1)^k\mu_{2k}^{\frac{s}{2}}}{\mu_n^s}.
\end{align}

One the other hand, since 
\begin{align}
    \frac{\langle 2k \rangle}{\langle n,n,2k\rangle}\frac{([k]!)^2}{[2k]!}=(-1)^{n+k}\frac{[2k+1]([n]!)^2}{[n+k+1]![n-k]!}, 
\end{align}
combining the last two formulas together, we prove Lemma \ref{Lemma-2n2n}.  
\end{proof}

\begin{lemma} \label{lemma-Rkj}
We have the following formula 
\begin{subnumcases}
{\raisebox{-45pt}{
\includegraphics[width=120 pt]{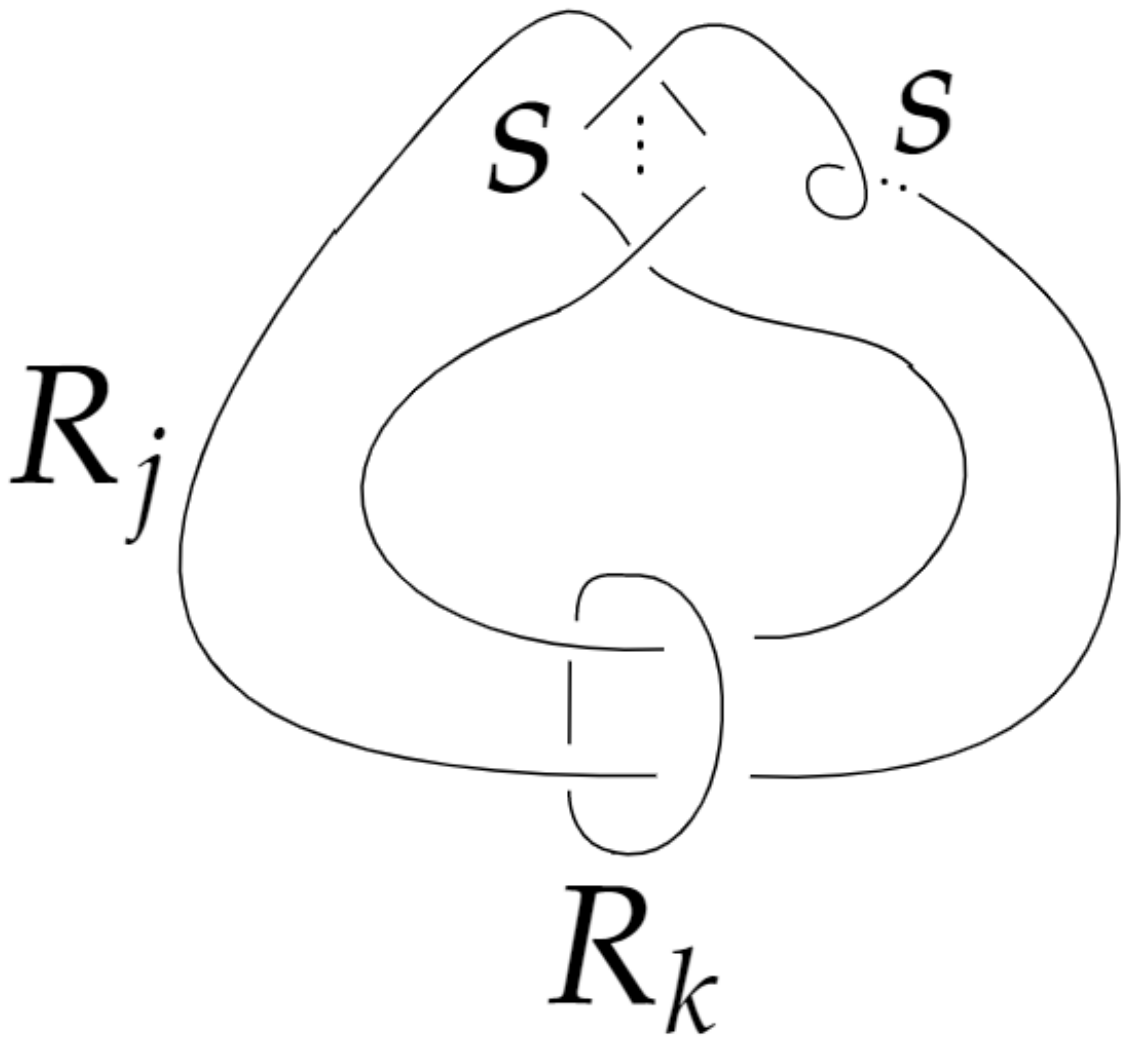}}=}
0, &$k\neq j$, \nonumber \\\nonumber
\frac{\{2j+1\}!(\{j\})^2}{\{1\}}\tilde{c}_{j,\frac{s}{2}}, & $k=j$.
\end{subnumcases}
\makeatletter\let\@alph\@@@alph\makeatother
where
\begin{align}
    \tilde{c}_{j,\frac{s}{2}}=\sum_{l=0}^j\frac{\mathfrak{q}^{l(l+1)s}\{2l+1\}}{\{j+l+1\}!\{j-l\}!}. 
\end{align}
\end{lemma}
\begin{proof}
Since each component of the above link in Lemma \ref{lemma-Rkj} is a zero-framed unknot having a spanning disk pierced twice by the other component, by using formula (\ref{formula-Rke2i}), we obtain the statement for $k\neq j$. 

For $k=j$, using that $R_j-e_j$ has degree less than $j$, by Lemma \ref{Lemma-2n2n},
\begin{align}
&\raisebox{-45pt}{
\includegraphics[width=120 pt]{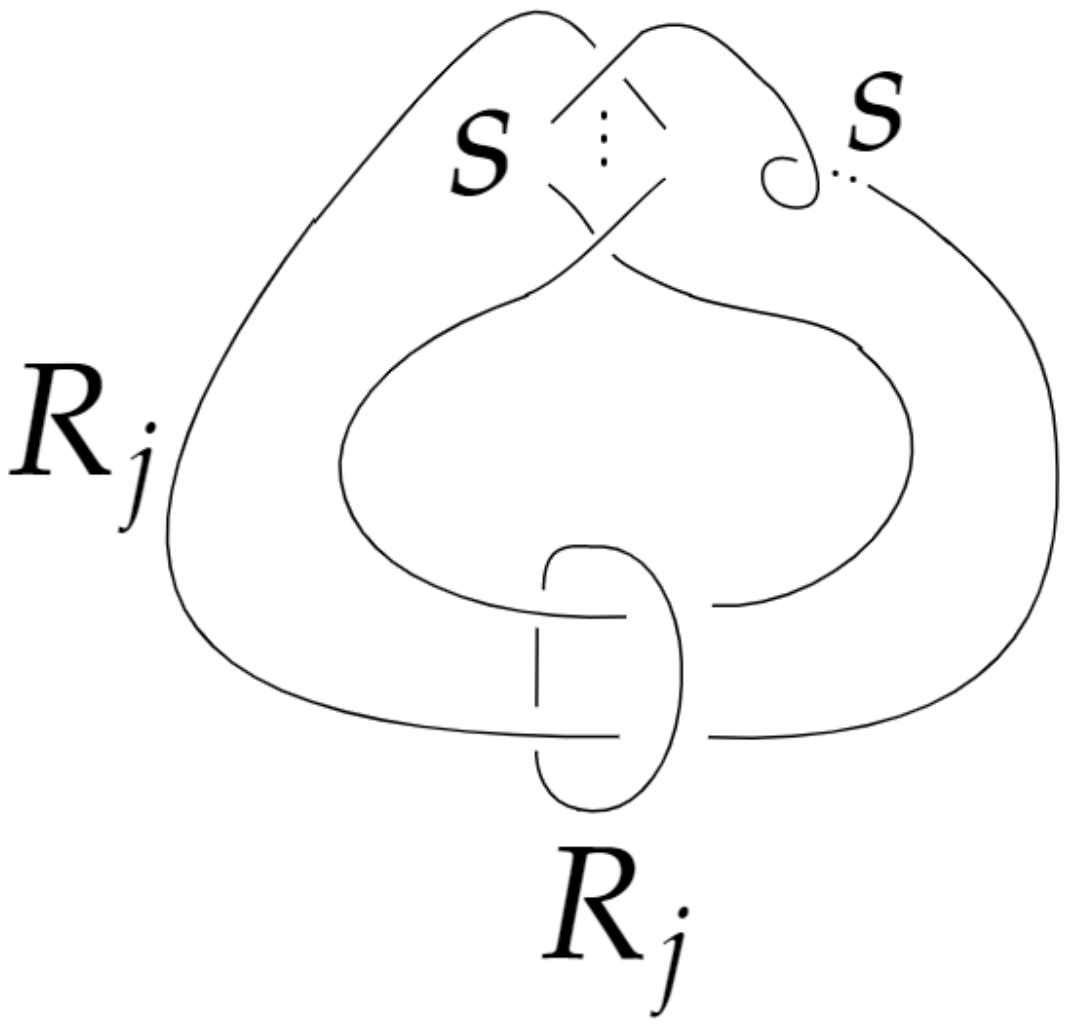}}=\raisebox{-45pt}{
\includegraphics[width=120 pt]{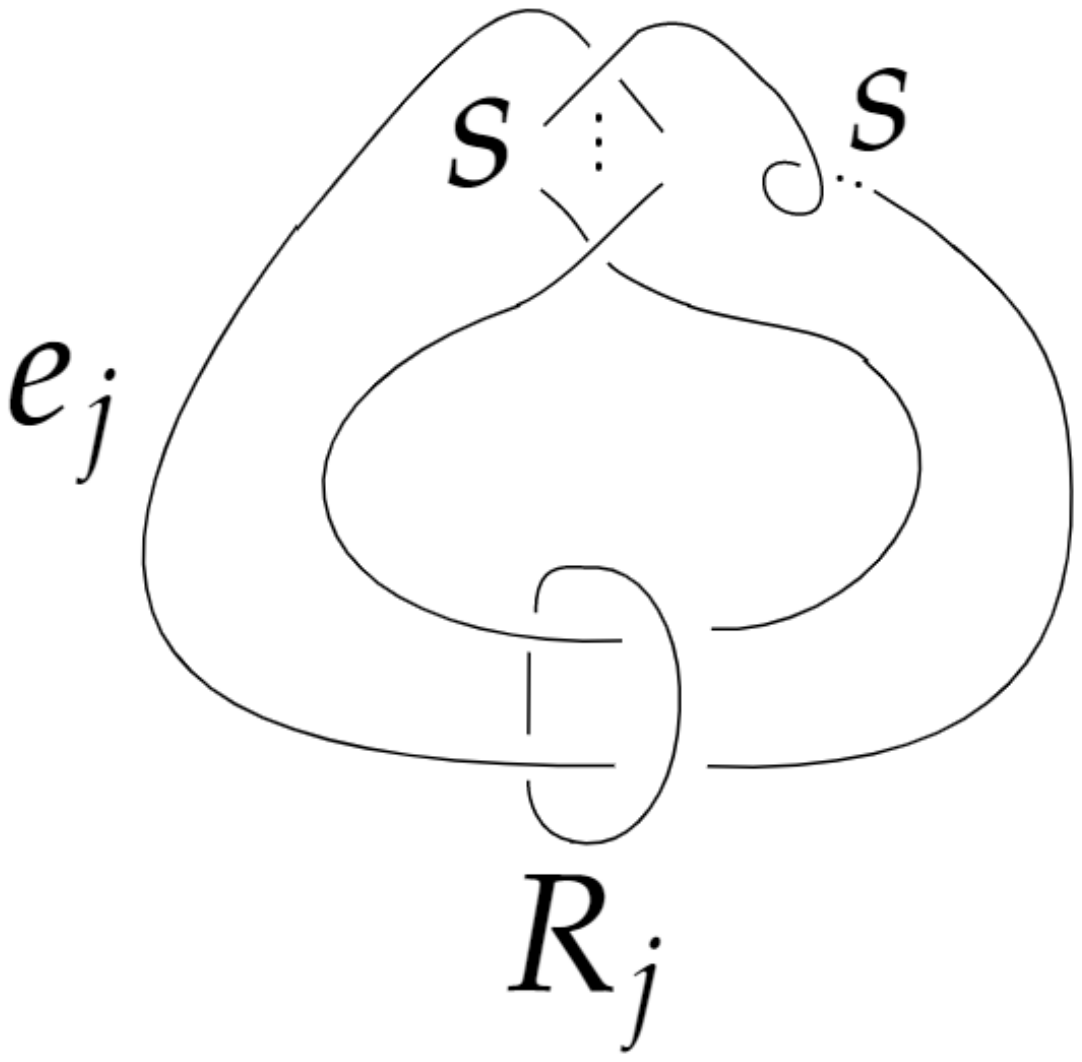}}=\mu_j^{s}\raisebox{-45pt}{
\includegraphics[width=120 pt]{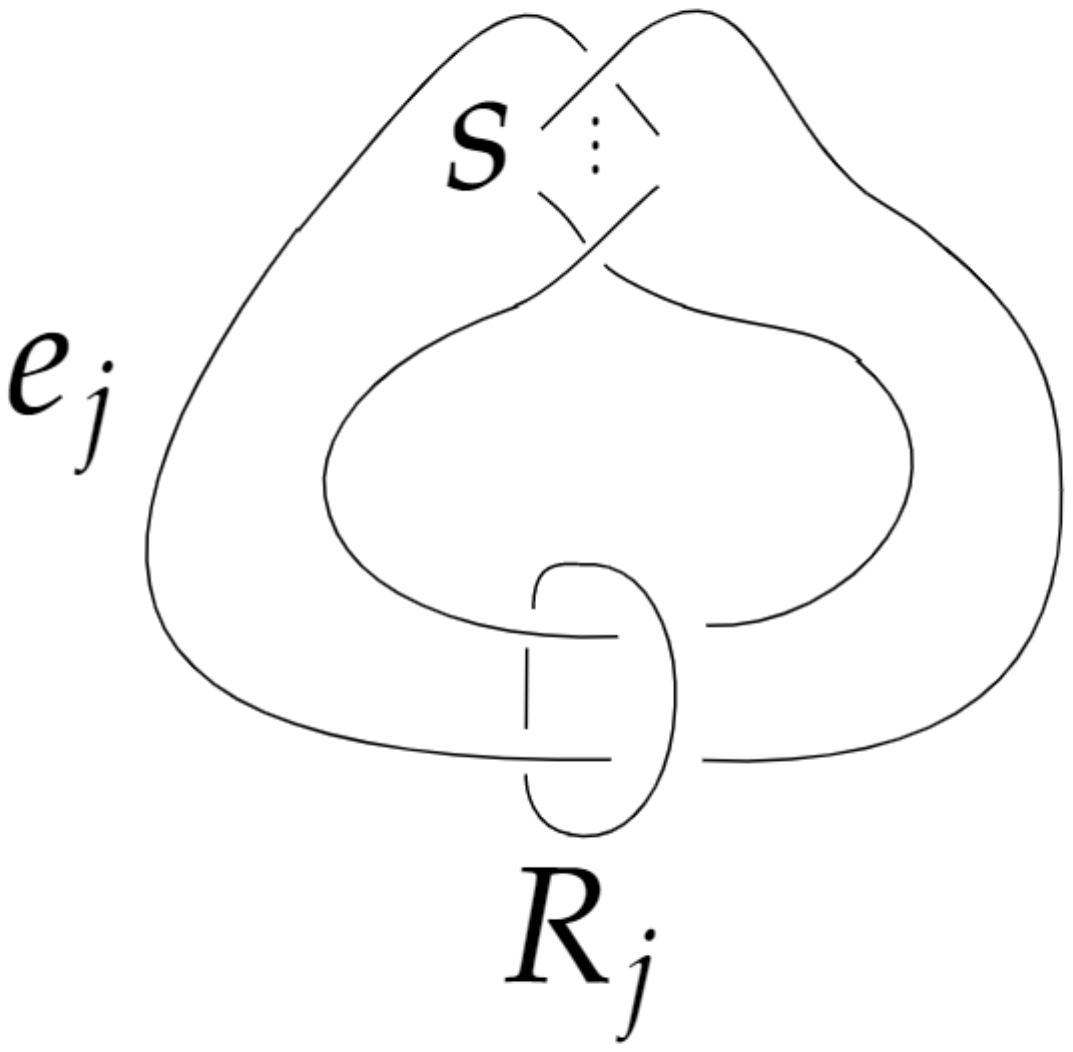}}\\\nonumber
&=\mu_j^{s}\raisebox{-45pt}{
\includegraphics[width=120 pt]{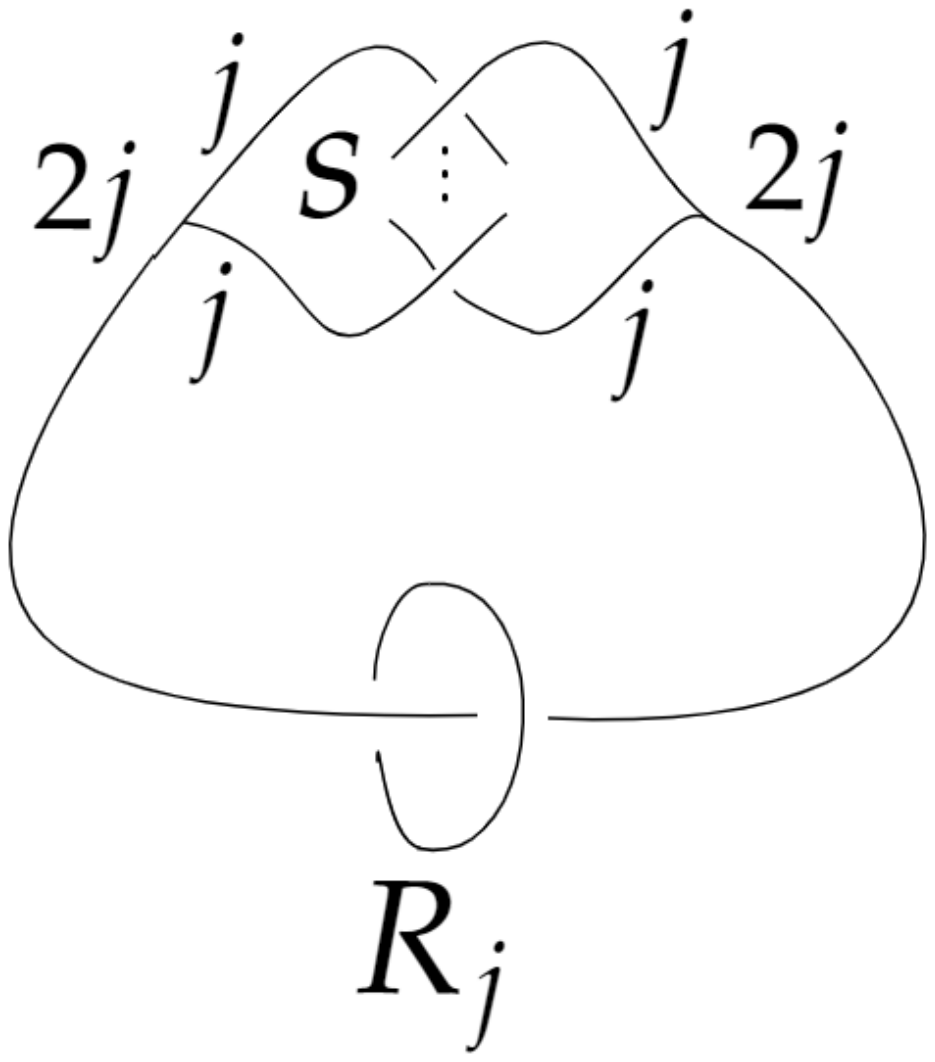}}=\sum_{l=0}^{j}(-1)^j\mu_{2l}^\frac{s}{2}\frac{[2l+1]([j]!)^2}{[j+l-1]![j-l]!}\raisebox{-45pt}{
\includegraphics[width=120 pt]{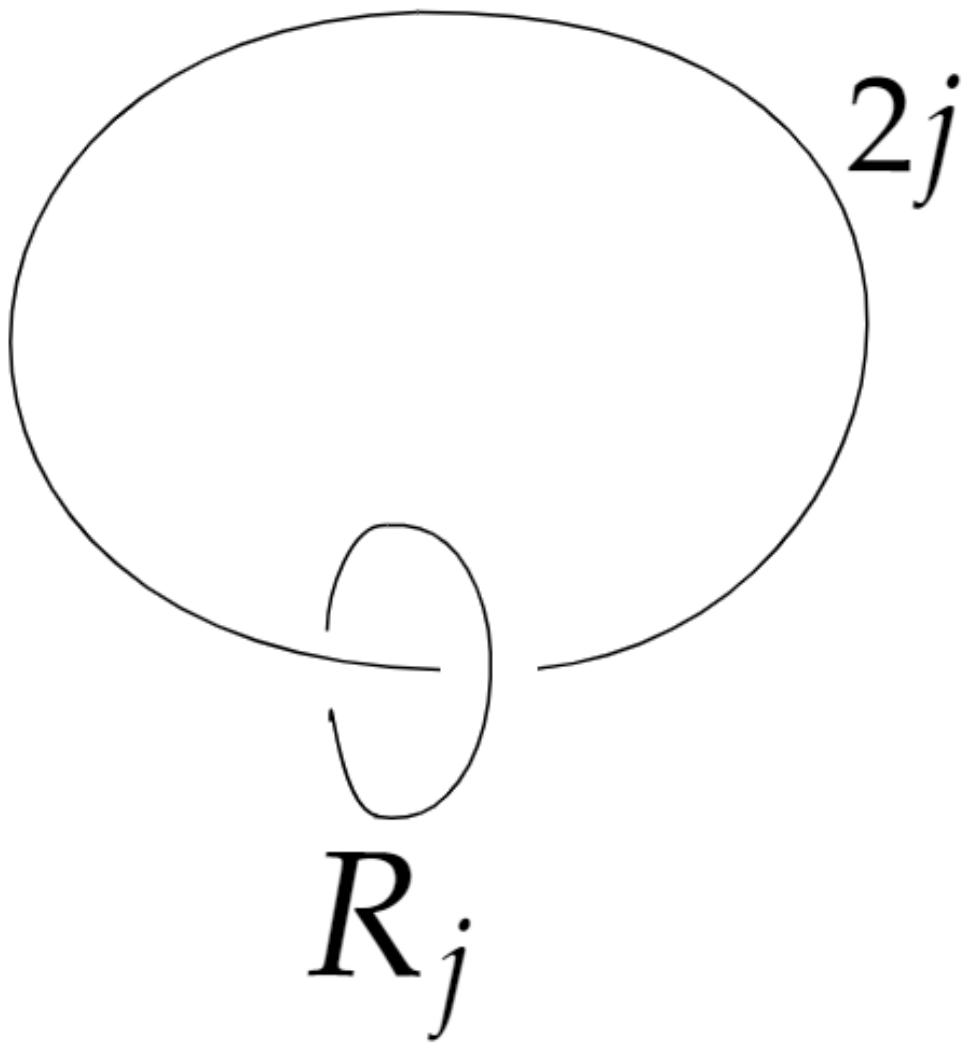}}\\\nonumber
&
=\frac{(\{2j+1\}!(\{j\}!)^2}{\{1\}}\sum_{l=0}^j\frac{\mathfrak{q}^{l(l+1)s}\{2l+1\}}{\{k+l+1\}!\{k-l\}!}.
\end{align}
We finish the proof of Lemma \ref{lemma-Rkj}. 
\end{proof}
For convenience, we introduce the notation 
\begin{align}
  \tilde{c}'_{j,\frac{s}{2}}=\{j\}!\tilde{c}_{j,\frac{s}{2}}.
\end{align}

Now, we can give the proof of our main theorem.

\begin{theorem}
The normalized $N$-th colored Jones polynomial of $\mathcal{K}_{p,\frac{s}{2}}$ is given by 
\begin{align} 
J'_{N}(\mathcal{K}_{p,\frac{s}{2}};\mathfrak{q})=\sum_{k=0}^{N-1}H_{k}(\mathcal{K}_{p,\frac{s}{2}};\mathfrak{q})\frac{\{N+k\}!}{\{N-1-k\}!\{N\}},
\end{align}
where
\begin{align} \label{formula-Hk}
H_{k}(\mathcal{K}_{p,\frac{s}{2}};\mathfrak{q})=(-1)^k\sum_{j=0}^kd_{k,j,p}c'_{j,p}
\tilde{c}'_{j,\frac{s}{2}},
\end{align}
with
\begin{align}
d_{k,j,p}&=\sum_{i=j}^k(-1)^{i+j}\mathfrak{q}^{-2pi(i+2)}\frac{\{2i+2\}\{i+1+j\}!}{\{k+i+2\}!\{k-i\}!\{i-j\}!},\label{formula-dkjp}\\
c'_{j,p}&=\{j\}!\sum_{l=0}^j\frac{(-1)^l\mathfrak{q}^{2pl(l+1)}\{2l+1\}}{\{j+l+1\}!\{j-l\}!}, \label{formula-c'jp}\\
\tilde{c}'_{j,\frac{s}{2}}&=\{j\}!\sum_{l=0}^j\frac{\mathfrak{q}^{sl(l+1)}\{2l+1\}}{\{j+l+1\}!\{j-l\}!}. \label{formula-tildecj}
\end{align}
\end{theorem}

\begin{proof}
First, note that the framing zero presentation of the knot $\mathcal{K}_{p,\frac{s}{2}}$ is given by 
\begin{align*}
\raisebox{-45pt}{
\includegraphics[width=120 pt]{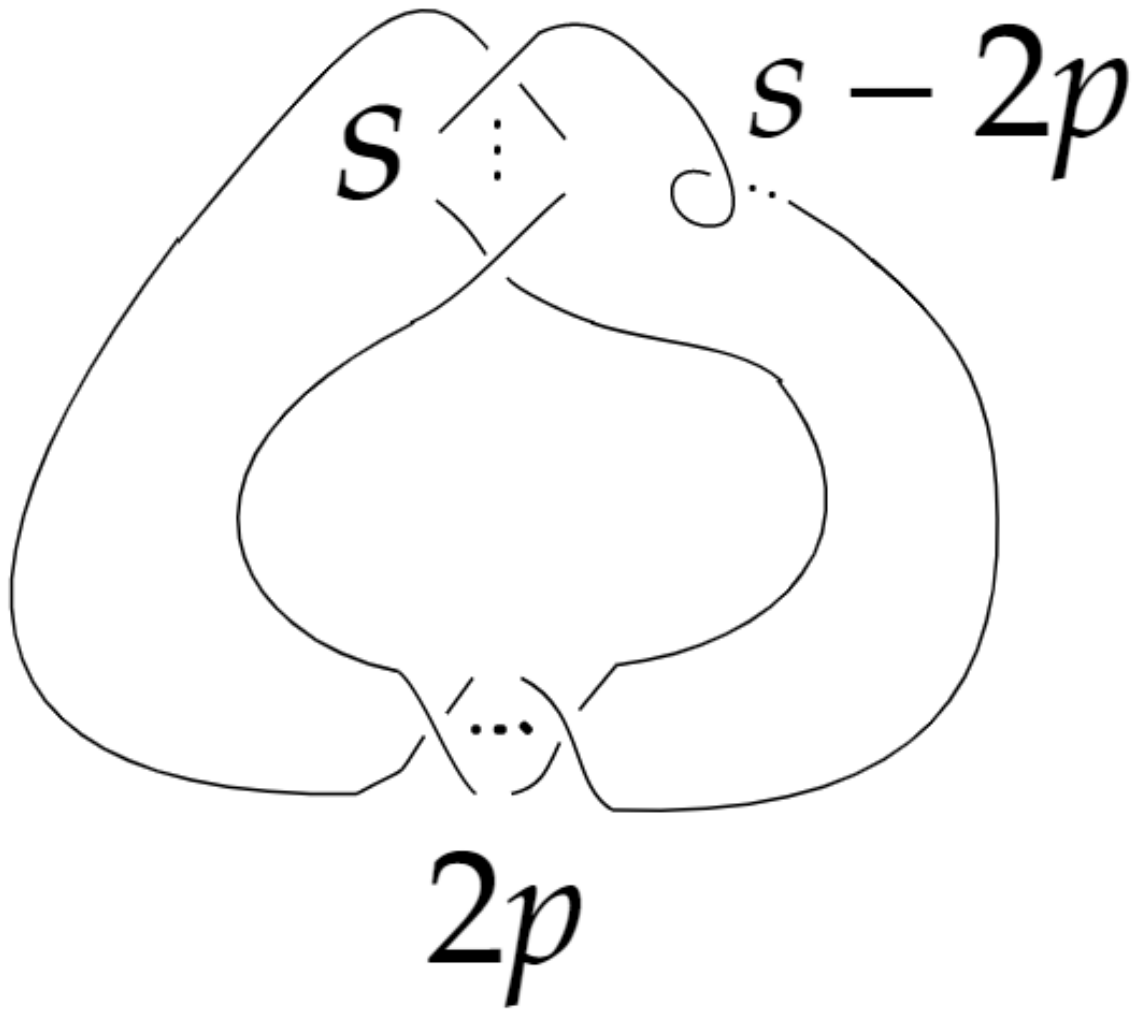}}= \raisebox{-45pt}{
\includegraphics[width=120 pt]{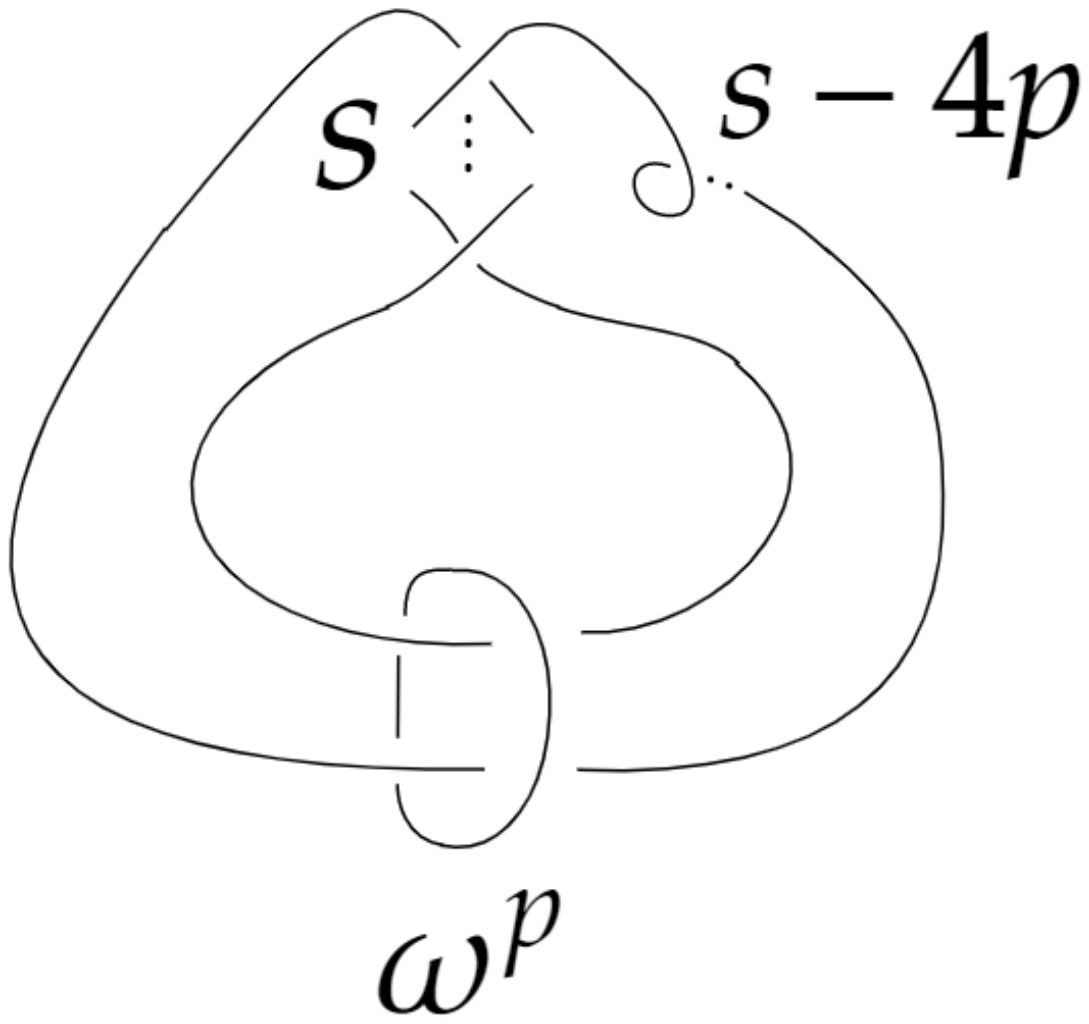}}.  
\end{align*}

Then, the colored Jones polynomial of $\mathcal{K}_{p,\frac{s}{2}}$ can be computed by the following diagram calculations.

\begin{align*}
J_{N}(\mathcal{K}_{p,\frac{s}{2}};\mathfrak{q})&=\raisebox{-45pt}{
\includegraphics[width=120 pt]{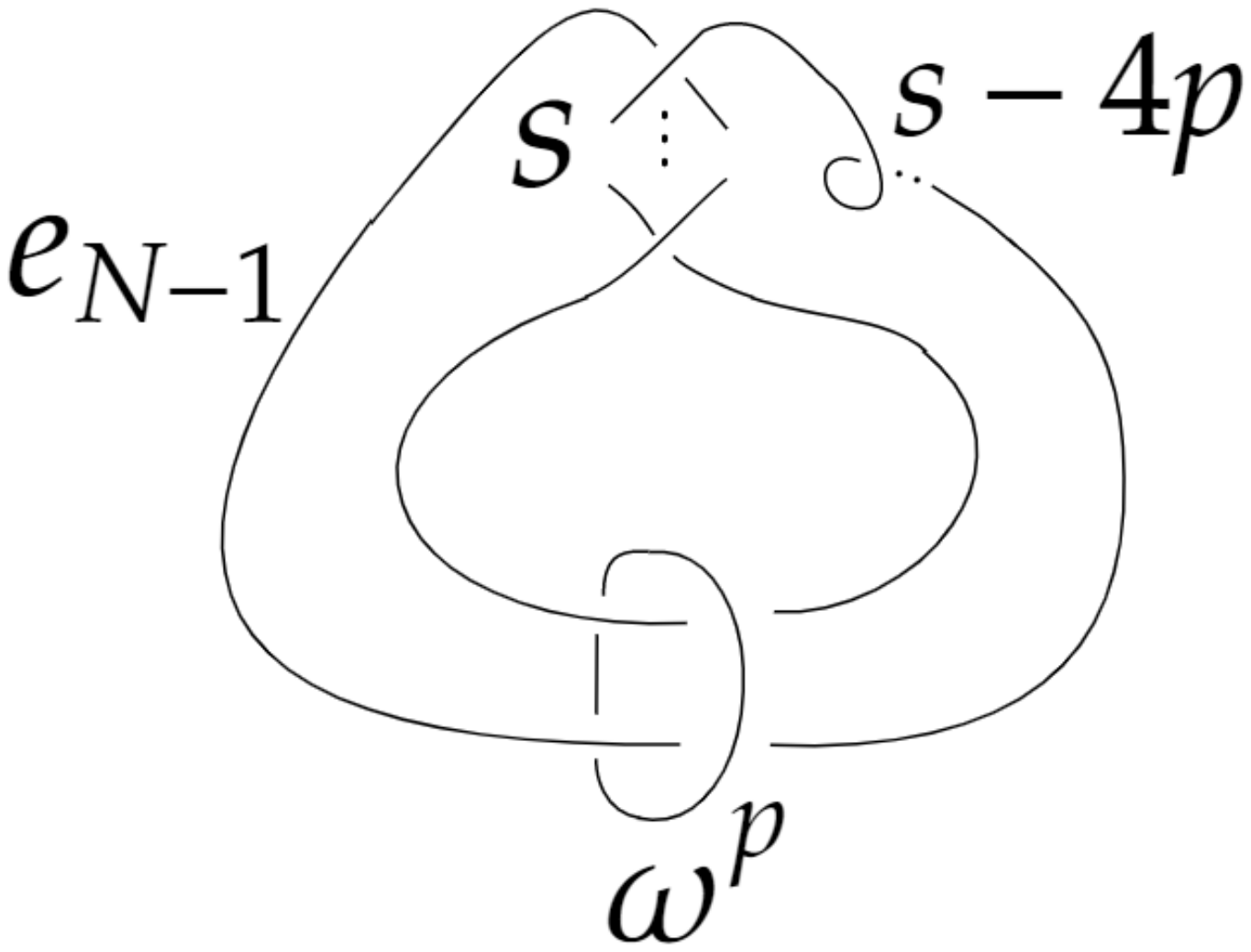}}\\
&=\sum_{k=0}^{N-1}
(-1)^{N-1-k}\left[
\begin{array}{@{\,}c@{\,}}N+k \\ N-1-k \end{array} \right]  \raisebox{-45pt}{\includegraphics[width=120
pt]{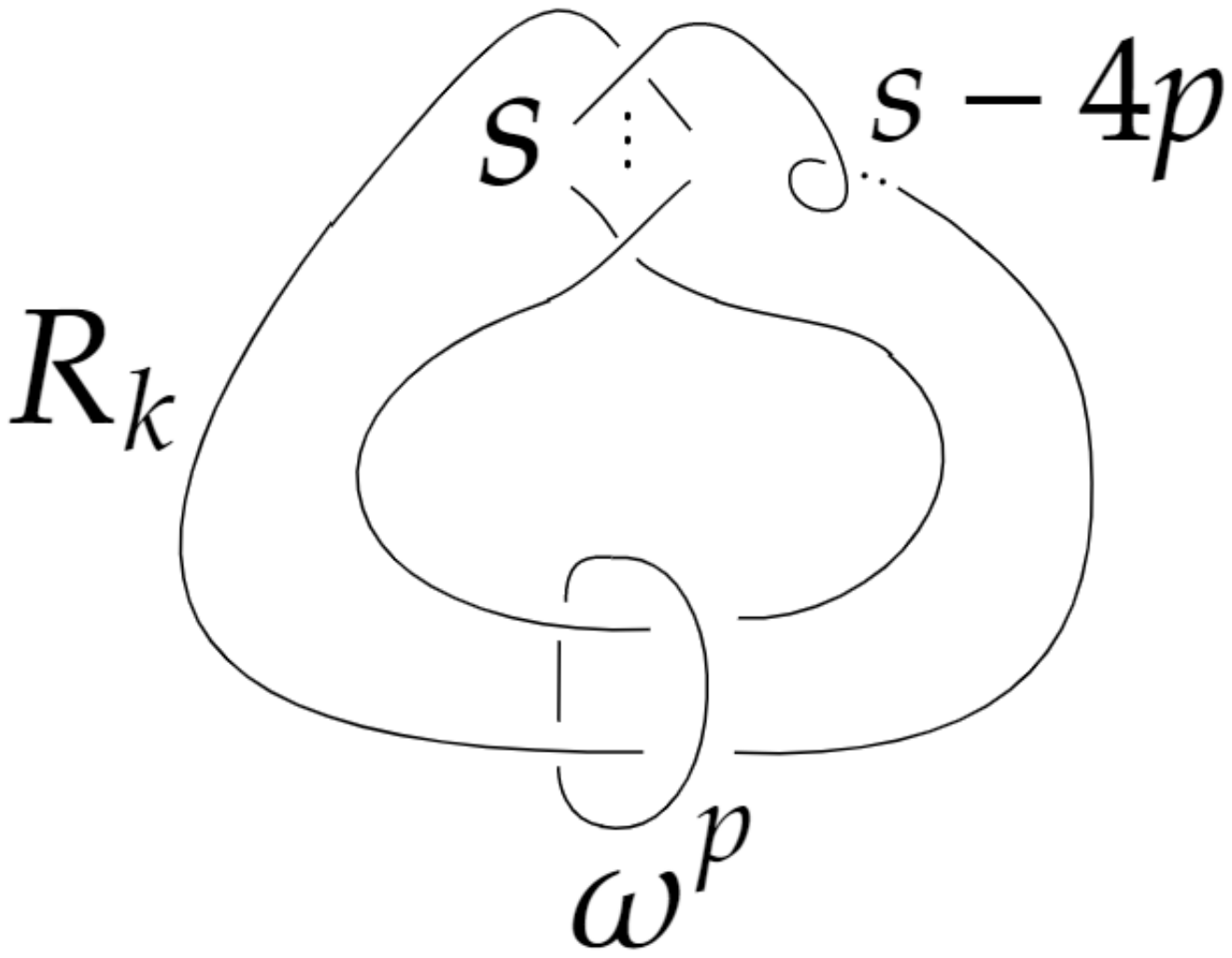}}\\ &=\sum_{k=0}^{N-1}
(-1)^{N-1-k}\left[
\begin{array}{@{\,}c@{\,}}N+k \\ N-1-k \end{array} \right]\sum_{j=0}^kd_{k,j}^{(-4p)}  \raisebox{-45pt}{\includegraphics[width=120
pt]{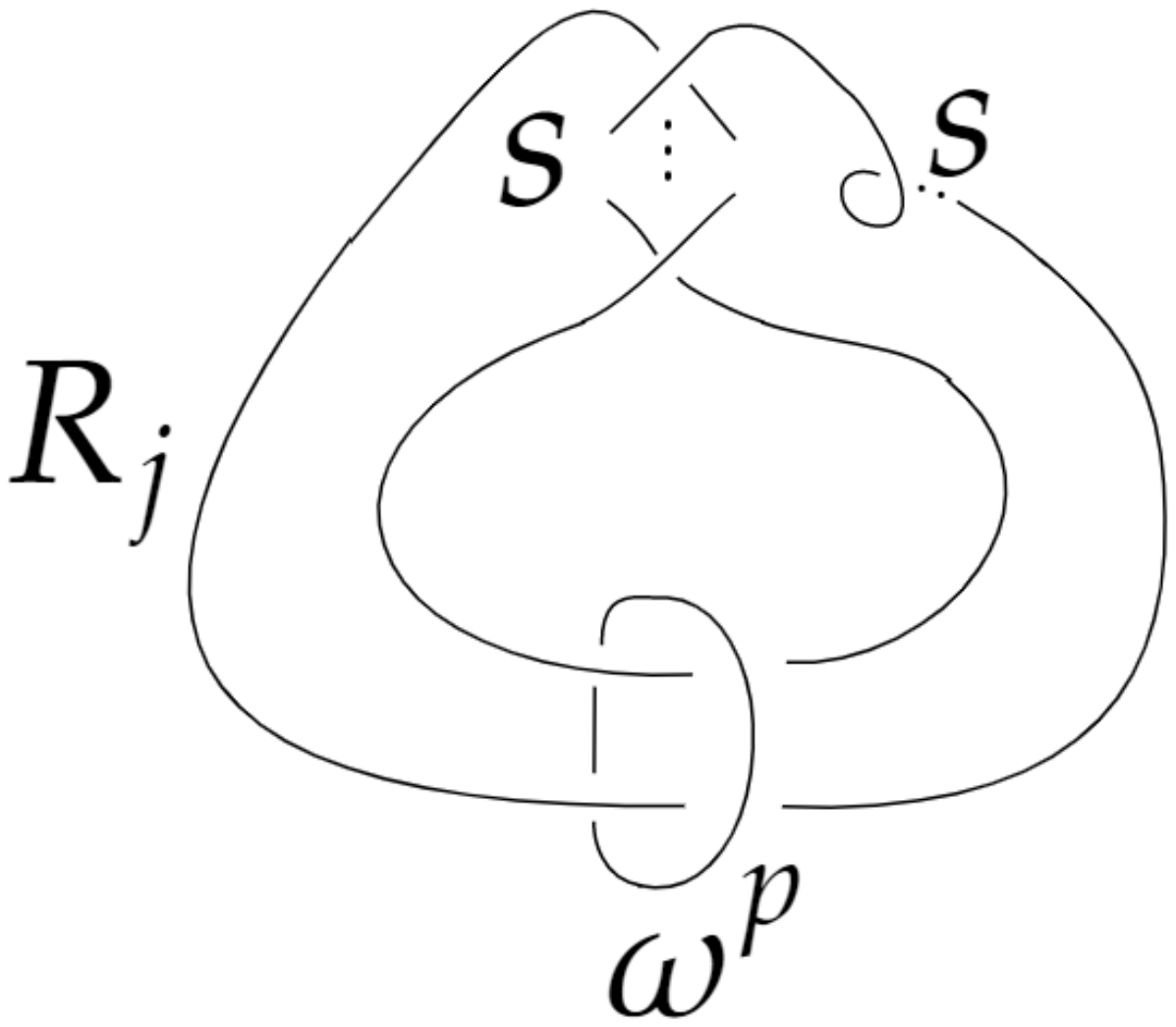}}\\ &=\sum_{k=0}^{N-1}
(-1)^{N-1-k}\left[
\begin{array}{@{\,}c@{\,}}N+k \\ N-1-k \end{array} \right]\sum_{j=0}^kd_{k,j}^{(-4p)}c_{j,p}  \raisebox{-45pt}{\includegraphics[width=120
pt]{RjRj.pdf}}
\end{align*}
By using Lemmas \ref{lemma-djk} and \ref{lemma-Rkj}, we obtain

\begin{align*}
&\sum_{j=0}^kd_{k,j}^{(-4p)}c_{j,p}  \raisebox{-45pt}{\includegraphics[width=120
pt]{RjRj.pdf}}\\\nonumber
&=\sum_{j=0}^k\sum_{i=j}^k(-1)^{i+j}\mathfrak{q}^{-2p(i+2)i}\frac{\{2k+1\}!\{2i+2\}!\{i+1+j\}!(\{j\}!)^2}{\{k+i+2\}!\{k-i\}!\{i-j\}!\{1\}}c_{j,p}
\tilde{c}_{j,\frac{s}{2}}.
\end{align*}

Therefore, we have
\begin{align}
    J'_{N}(\mathcal{K}_{p,\frac{s}{2}};\mathfrak{q})&=\frac{(-1)^{N-1}\{1\}}{\{N\}}J_{N}(\mathcal{K}_{p,\frac{s}{2}};\mathfrak{q}) \\\nonumber
    &=\sum_{k=0}^{N-1}H_{k}(\mathcal{K}_{p,\frac{s}{2}};\mathfrak{q})\frac{\{N+k\}!}{\{N-1-k\}!\{N\}},
\end{align}
with 
\begin{align}
H_{k}(\mathcal{K}_{p,\frac{s}{2}};\mathfrak{q})=(-1)^k\sum_{j=0}^k\sum_{i=j}^k(-1)^{i+j}\mathfrak{q}^{-2pi(i+2)}\frac{\{2i+2\}\{i+1+j\}!(\{j\}!)^2}{\{k+i+2\}!\{k-i\}!\{i-j\}!}c_{j,p}
\tilde{c}_{j,\frac{s}{2}}.
\end{align}
\end{proof}

\section{Multiple sum expressions} \label{Section-Bailey}
Although Habiro's theorem arsserts that the coefficient  $H_{k}(\mathcal{K}_{p,\frac{s}{2}};\mathfrak{q})$ in the formula (\ref{formula-JN}) lies in the integral coefficient ring $\mathbb{Z}[\mathfrak{q}^{\pm 1}]$. It is not obvious to see this fact from the expression (\ref{formula-Hk0})  directly. For the double twist knot $\mathcal{K}_{p,r}$, by using the Kauffman bracket skein theory, Masbaum \cite{Mas03} obtained an multi-sum expression for  $c'_{k,p}$ (see also \cite{Hab08}) which implies that the coefficient $H_k(\mathcal{K}_{p,r};\mathfrak{q})$ lies in the ring $\mathbb{Z}[\mathfrak{q}^{\pm 1}]$.      

Actually, there is another technique named  Bailley chains \cite{And84} which can be used to derive Masbaum's multi-sum expression for $c'_{k,p}$ directly. In this section, we will first review the Bailey chains technique, and then derive the multi-sum expression for $c'_{k,p}, \tilde{c}'_{k,\frac{s}{2}}$ and $d_{k,j,p}$ respectively. Hence, we prove Theorem \ref{theorem-2}.

\subsection{Bailey chains}
We recall some basics of Bailey chains as illustrated in the appendix of \cite{NRZ12}. The details and proofs can be found in \cite{And84}. A pair of sequences $(a_k,b_k)_{k\geq 0}$ is called a Bailey pair if they are related by  
\begin{align} \label{formula-bn}
    b_k=\sum_{j=0}^k\frac{a_j}{(q;q)_{k-j}(xq;q)_{k+j}}
\end{align}
or equivalently 
\begin{align}
    a_k=\frac{1-xq^{2k}}{1-x}\sum_{j=0}^{k-j}q^{(k-j)(k-j-1)/2}\frac{(x;q)_{k+j}}{(q;q)_{k-j}}b_j.
\end{align}

Given a Bailey pair $(a_k,b_k)_{k\geq 0}$, one can construct a new Bailley pair $(a'_k,b'_k)_{k\geq 0}$ by 
\begin{align}
    a'_k=x^kq^{k^2}a_k, \ b'_k=\sum_{j=0}^k\frac{x^jq^{j^2}}{(q;q)}b_j, \ k=0,1,2,3,...
\end{align}
A sequence $(\alpha^{(1)},\beta^{(1)})\rightarrow (\alpha^{(2)},\beta^{(2)})\rightarrow \cdots \rightarrow (\alpha^{(l)},\beta^{(l)})\rightarrow \cdots$ obtained by applying iteratively is called a Bailey Chain. The key observation is that two Bailey pairs next to each other in a Bailey chain, 
$(\alpha_n^{(l-1)},\beta_n^{(l-1)})_{n\geq 0}$ and $(\alpha_n^{(l)},\beta_n^{(l)})_{n\geq 0}$, satisfy the special version of the Bailey's lemma
\begin{align} \label{formula-apidentity}
    \sum_{j=0}^k\frac{\alpha_j^{(l)}}{(q;q)_{k-j}(xq;q)_{k+j}}=\sum_{j=0}^k\frac{x^jq^{j^2}}{(q;q)_{k-j}}\sum_{k=0}^j\frac{\alpha_k^{(l-1)}}{(q;q)_{j-k}(xq;q)_{j+k}}.
\end{align}

Using the formula (\ref{formula-apidentity}) recursively, we obtain 
\begin{align} \label{formula-alphap}
    &(q;q)_{k_p}\sum_{k_{p-1}=0}^{k_p}\frac{\alpha_{k_{p-1}}^{(p)}}{(q;q)_{k_p-k_{p-1}}(xq;q)_{k_p+k_{p-1}}}\\\nonumber
    &=(q;q)_{k_p}\sum_{k_{p-1}}^{k_p}\frac{x^{k_{p-1}}q^{k_{p-1}^2}}{(q;q)_{k_p-k_{p-1}}}\sum_{k_{p-2}=0}^{k_{p-1}}\frac{a_{k_{p-2}}^{(p-1)}}{(q;q)_{k_{p-1}-k_{p-2}}(xq;q)_{k_{p-1}+k_{p-2}}}\\\nonumber
    &=\cdots\\\nonumber
    &=(q;q)_{k_p}\sum_{k_p\geq \cdots\geq k_1\geq 0}\prod_{i=1}^{p-1}x^{k_i}q^{k_i^2}\frac{1}{(q;q)_{k_p-k_{p-1}}\cdots (q;q)_{k_2-k_1}}\sum_{k_0=0}^{k_1}\frac{\alpha^{(1)}_{k_0}}{(q;q)_{k_1-k_0}(xq;q)_{k_1+k_0}}\\\nonumber
\end{align}

In particular, taking $x=q$ in (\ref{formula-alphap}) and by 
\begin{align}
    \alpha_{l}^{(p)}=q^{(l^2+l)(p-1)}\alpha^{(1)}_{l}
\end{align}
we obtain
\begin{align}   \label{formla-alpha1}
    &(q;q)_{k}\sum_{l=0}^{k}\frac{q^{(l^2+l)(p-1)}\alpha^{(1)}_{l}(1-q)}{(q;q)_{k-l}(q;q)_{k+l+1}}\\\nonumber
    &=\sum_{k=k_p\geq \cdots\geq k_1\geq 0}\prod_{i=1}^{p-1}q^{k_i^2+k_i}\frac{(q;q)_{k}}{(q;q)_{k_p-k_{p-1}}\cdots (q;q)_{k_2-k_1}}\sum_{k_0=0}^{k_1}\frac{\alpha^{(1)}_{k_0}}{(q;q)_{k_1-k_0}(q;q)_{k_1+k_0+1}}\\\nonumber
\end{align}

\subsection{Multiple sum expression for the coefficient $c_{k,p}$}
Now we consider the Bailey pair \cite{And84} given by 
\begin{align}  \label{formula-alpha1k}
    \alpha^{(1)}_k=(-1)^kq^{k^2+k+\frac{k(k-1)}{2}}\frac{1-q^{2k+1}}{1-q},  \ \beta^{(1)}_k=\frac{1}{(q;q)_k}. 
\end{align}
By definition, $(\alpha^{(1)}_k,\beta^{(1)}_k)$ satisfies the identity (\ref{formula-bn}), which gives
\begin{align}  \label{formula-alpha1l}
\sum_{l=0}^k\frac{\alpha^{(1)}_{l}}{(q;q)_{k-l}(q;q)_{k+l+1}}=\frac{1}{(q;q)_k}.
\end{align}

Substituting  (\ref{formula-alpha1k}) and (\ref{formula-alpha1l}) to formula (\ref{formla-alpha1}), we obtain
\begin{align}
    (q;q)_{k}\sum_{l=0}^{k}\frac{(-1)^lq^{l(l+1)p+\frac{l(l-1)}{2}}(1-q^{2l+1})}{(q;q)_{k-l}(q;q)_{k+l+1}}
    =\sum_{k=k_p\geq \cdots \geq k_1\geq 0}\prod_{i=1}^{p-1}q^{k_i^{2}+k_i}\left[
\begin{array}{@{\,}c@{\,}}k_{i+1} \\ k_i \end{array} \right]_q.
\end{align}

By using the formula (\ref{formula-kckp}) and $q=\mathfrak{q}^2$,  one gets 
\begin{align}
c'_{k,p}=(-1)^kq^{\frac{k(k+3)}{4}}\sum_{k=k_p\geq \cdots \geq k_1\geq 0}\prod_{i=1}^{p-1}q^{k_i^{2}+k_i}\left[
\begin{array}{@{\,}c@{\,}}k_{i+1} \\ k_i \end{array} \right]_q\in \mathbb{Z}[\mathfrak{q}^{\pm 1}].    
\end{align}

\subsection{Multiple sum expression for the coefficient $c_{k,\frac{s}{2}}$}
Since $s$ is odd, we can write $s=2m-1$, 
By the expression (\ref{formula-tildecj}) for $\tilde{c}_{k,\frac{s}{2}}$, we have
\begin{align}
    \tilde{c}'_{k,\frac{2m-1}{2}}=(-1)^kq^{\frac{k^2+3k}{4}}\sum_{l=0}^{k}q^{ml^2+(m-1)l}\frac{(1-q^{2l+1})(q;q)_k}{(q;q)_{k+l+1}(q;q)_{k-l}}.
\end{align}

In this case, we start with the following Bailey pair (see formula (4.12) in \cite{War03})
\begin{align}
    \alpha^{(1)}_k=\frac{q^{k^2}(1-q^{2k+1})}{1-q}, \ \beta^{(1)}_k=\frac{1}{(q;q)^2_k}.
\end{align}

Substituting it to  formula (\ref{formla-alpha1}),  we obtain 
\begin{align}
    \sum_{l=0}^{k}q^{ml^2+(m-1)l}\frac{(1-q^{2l+1})(q;q)_k}{(q;q)_{k+l+1}(q;q)_{k-l}}=\sum_{k=k_m\geq \cdots\geq k_1\geq 0 }\frac{1}{(q;q)_{k_1}}\prod_{i=1}^{m-1}q^{k_i^2+k_i}\left[
\begin{array}{@{\,}c@{\,}}k_{i+1} \\ k_i \end{array} \right]_q.
\end{align}

Hence
\begin{align}
    \tilde{c}'_{k,\frac{2m-1}{2}}\in \frac{1}{\{k\}!}\mathbb{Z}[\mathfrak{q}^{\pm 1}]. 
\end{align}

\subsection{Multiple sum expression for the coefficient $d_{k,j}$}
For $0\leq j\leq k$, 
\begin{align}
    d_{k,j,p}=(-1)^j\sum_{i=j}^k(-1)^i\mathfrak{q}^{-2pi(i+2)}\frac{\{2i+2\}\{i+1+j\}\{i+j\}\cdots\{i+1-j\}}{\{k+i+2\}!\{k-i\}!}.
\end{align}

Introduce the index $l$ by $i=l+j$, the formula  can be  written as  
\begin{align}
    d_{k,j,p}=\sum_{l=0}^{k-j}(-1)^l\mathfrak{q}^{-2p(l+j)(l+j+2)}\frac{\{2l+2j+2\}\{l+2j+1\}\{l+2j\}\cdots\{l+1\}}{\{k+l+j+2\}!\{k-l-j\}!}.
\end{align}

Using the $q$-Pochhammer symbols, we obtain 
\begin{align} \label{formula-dkjp2}
     d_{k,j,p}=q^{\frac{k(k+3)-j(j+3)}{2}-pj(j+2)}\sum_{l=0}^{k-j}(-1)^lq^{-p(l^2+(2j+2)l)+\frac{l(l-1)}{2}}\frac{(1-q^{2(l+j+1)})(q;q)_{l+2j+1}}{(q;q)_{k+l+j+2}(q;q)_{k-l-j}(q;q)_{l}}.
\end{align}

Taking $x=q^{2j+2}$, replacing $k$ by $(k-j)$ in (\ref{formula-alphap}), and using
\begin{align}
    \alpha_{l}^{(p)}=x^lq^{l^2}=q^{(l^2+(2j+2)l)(p-1)}\alpha^{(1)}_{l},
\end{align}
we obtain
\begin{align}   \label{formla-alpha1new}
    &(q;q)_{k-j}\sum_{l=0}^{k-j}\frac{q^{(l^2+(2j+2)l)(p-1)}\alpha^{(1)}_{l}}{(q;q)_{k-j-l}(q^{2j+3};q)_{k-j+l}}\\\nonumber
    &=\sum_{k-j=k_p\geq \cdots\geq k_1\geq 0}\prod_{i=1}^{p-1}q^{k_i^2+k_i}\frac{(q;q)_{k-j}}{(q;q)_{k_p-k_{p-1}}\cdots (q;q)_{k_2-k_1}}\sum_{k_0=0}^{k_1}\frac{\alpha^{(1)}_{k_0}}{(q;q)_{k_1-k_0}(q^{2j+3};q)_{k_1+k_0}}\\\nonumber
\end{align}

We consider the following Bailey pair \cite{And84},  
\begin{align}
    \alpha^{(1)}_l&=(-1)^lx^lq^{l^2+\frac{l(l-1)}{2}}\frac{1-xq^{2l}}{1-x}\frac{(x;q)_{l}}{(q;q)_l}  \ \ (\text{Taking $x=q^{2j+2}$})  \\\nonumber
    &=(-1)^lq^{l^2+(2j+2)l+\frac{l(l-1)}{2}}\frac{(1-q^{2l+2j+2})(q;q)_{2j+l+1}}{(q;q)_l(q;q)_{2j+2}} \\\nonumber
\beta^{(1)}_l&=\frac{1}{(q;q)_l}.
\end{align}

Formula (\ref{formla-alpha1}) gives
\begin{align}  
\sum_{k_0=0}^{k_1}\frac{\alpha^{(1)}_{k_0}}{(q;q)_{k_1-k_0}(q^{2j+3};q)_{k_1+k_0}}=\beta^{(1)}_{k_1}=\frac{1}{(q;q)_{k_1}}.
\end{align}
Substituting them back to formula (\ref{formla-alpha1new}), we obtain
\begin{align}  
    &(q;q)_{k-j}\sum_{l=0}^{k-j}(-1)^l\frac{q^{p(l^2+(2j+2)l)+\frac{l(l-1)}{2}}}{(q;q)_{k-j-l}(q^{2j+3};q)_{k-j+l}}\frac{(1-q^{2l+2j+2})(q;q)_{2j+l+1}}{(q;q)_l(q;q)_{2j+2}}\\\nonumber
    &=\sum_{k-j=k_p\geq \cdots\geq k_1\geq 0}\prod_{i=1}^{p-1}q^{k_i^2+k_i}\frac{(q;q)_{k-j}}{(q;q)_{k_p-k_{p-1}}\cdots (q;q)_{k_2-k_1}}\sum_{k_0=0}^{k_1}\frac{\alpha^{(1)}_{k_0}}{(q;q)_{k_1-k_0}(q^{2j+3};q)_{k_1+k_0+1}}\\\nonumber
    &=\sum_{k-j=k_p\geq \cdots\geq k_1\geq 0 }\prod_{i=1}^{p-1}q^{k_i^2+k_i}\left[
\begin{array}{@{\,}c@{\,}}k_{i+1} \\ k_i \end{array} \right]_q.
\end{align}

Comparing to the formula (\ref{formula-dkjp2}), for $p>0$,  we obtain
\begin{align}
d_{k,j,-p}=q^{\frac{k(k+3)-j(j+3)}{2}+pj(j+2)}\frac{1}{(q;q)_{k-j}}\sum_{k-j=k_p\geq \cdots\geq k_1\geq 0 }\prod_{i=1}^{p-1}q^{k_i^2+k_i}\left[
\begin{array}{@{\,}c@{\,}}k_{i+1} \\ k_i \end{array} \right]_q.
\end{align}
Since 
\begin{align}
    (1/q;1/q)_k=(-1)^kq^{-\frac{k(k+1)}{2}}, \ \ \left[ \begin{array}{@{\,}c@{\,}}k \\ l \end{array} \right]_{1/q}=q^{l^2-lk}\left[ \begin{array}{@{\,}c@{\,}}k \\ l \end{array} \right]_{q},
\end{align}
we get
\begin{align}
    d_{k,j,p}(q)&=d_{k,j,-p}(1/q)\\\nonumber
    &=(-1)^{k-j}q^{(j+1)(j-k)-pj(j+2)}\frac{1}{(q;q)_{k-j}}\sum_{k-j=k_p\geq \cdots\geq k_1\geq 0 }\prod_{i=1}^{p-1}q^{-k_ik_{i+1}-k_i}\left[
\begin{array}{@{\,}c@{\,}}k_{i+1} \\ k_i \end{array} \right]_q.
\end{align}

\section{Appendix: corrected form of Walsh's formula}
K. Walsh first computed an formula for the colored Jones polynomial of $\mathcal{K}_{p,\frac{s}{2}}$ which is not in the cyclotomic expansion form (cf. Corollary 4.2.4 in \cite{Wal14}). But the original version of Walsh's formula (i.e. Corollary 4.2.4 in \cite{Wal14}) is not right due to the incorrect using of framing change. 

Although the corrected form of Walsh's formula has been given in \cite{LO19}( cf. formula (5.13) in \cite{LO19}), for the completeness, we provide the derivation of this corrected formula in this appendix.
\begin{theorem}
The (normalized) colored Jones polynomial for double twist knot
$\mathcal{K}_{p,\frac{s}{2}}$ is given by
\begin{align}
J'_{N}(\mathcal{K}_{p,\frac{s}{2}};\mathfrak{q})=q^{-2p(N^2-1)}\sum_{k=0}^{N-1}(-1)^kc'_{k,p}\tilde{c}'_{k,\frac{s}{2}}\frac{\{N+k\}!}{\{N-1-k\}!\{N\}}.
\end{align}
\end{theorem}

\begin{proof}
 Since
\begin{align*}
\raisebox{-45pt}{
\includegraphics[width=120 pt]{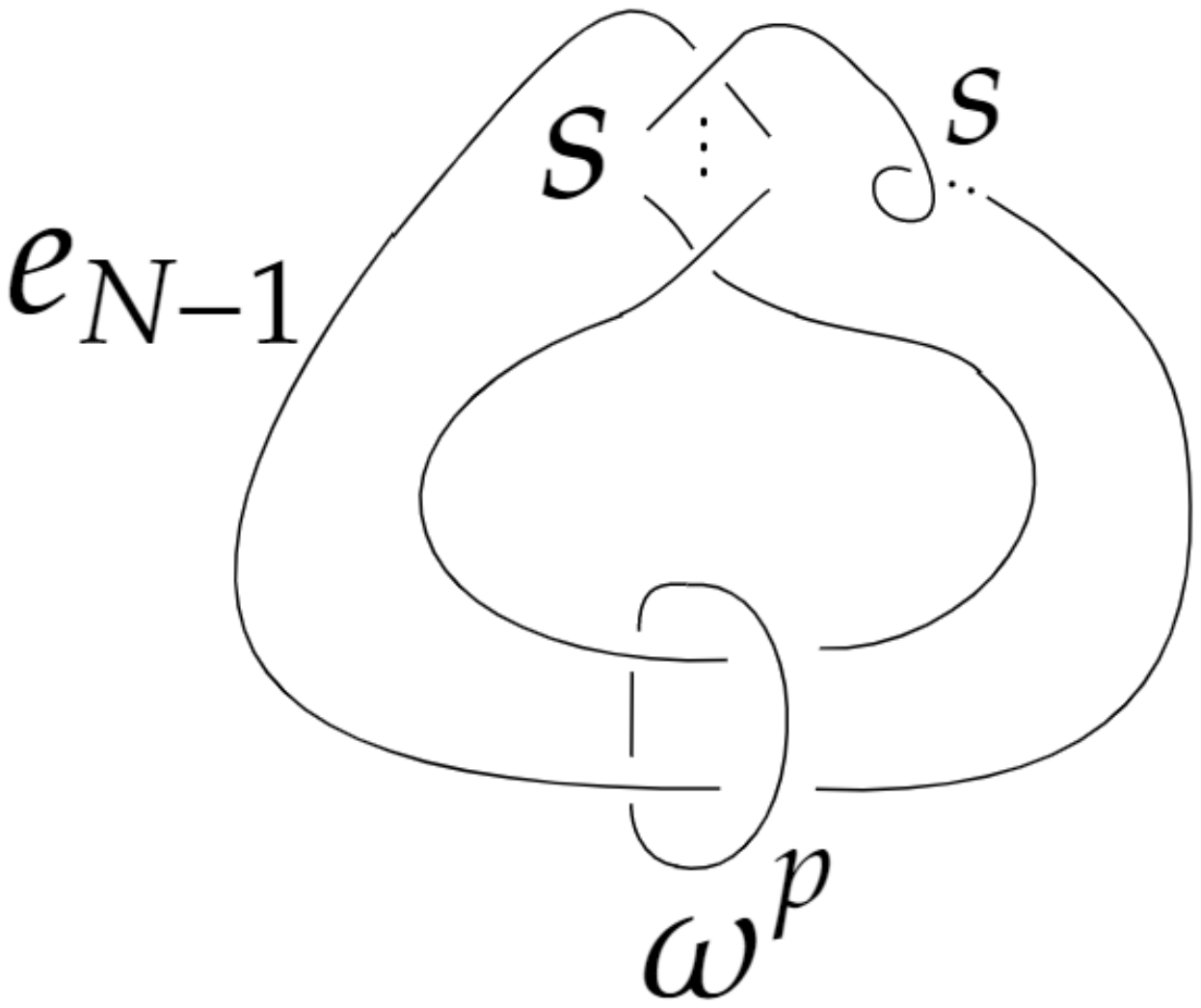}}&=\sum_{k=0}^{N-1}
(-1)^{N-1-k}\left[
\begin{array}{@{\,}c@{\,}}N+k \\ N-1-k \end{array} \right]  \raisebox{-45pt}{\includegraphics[width=120
pt]{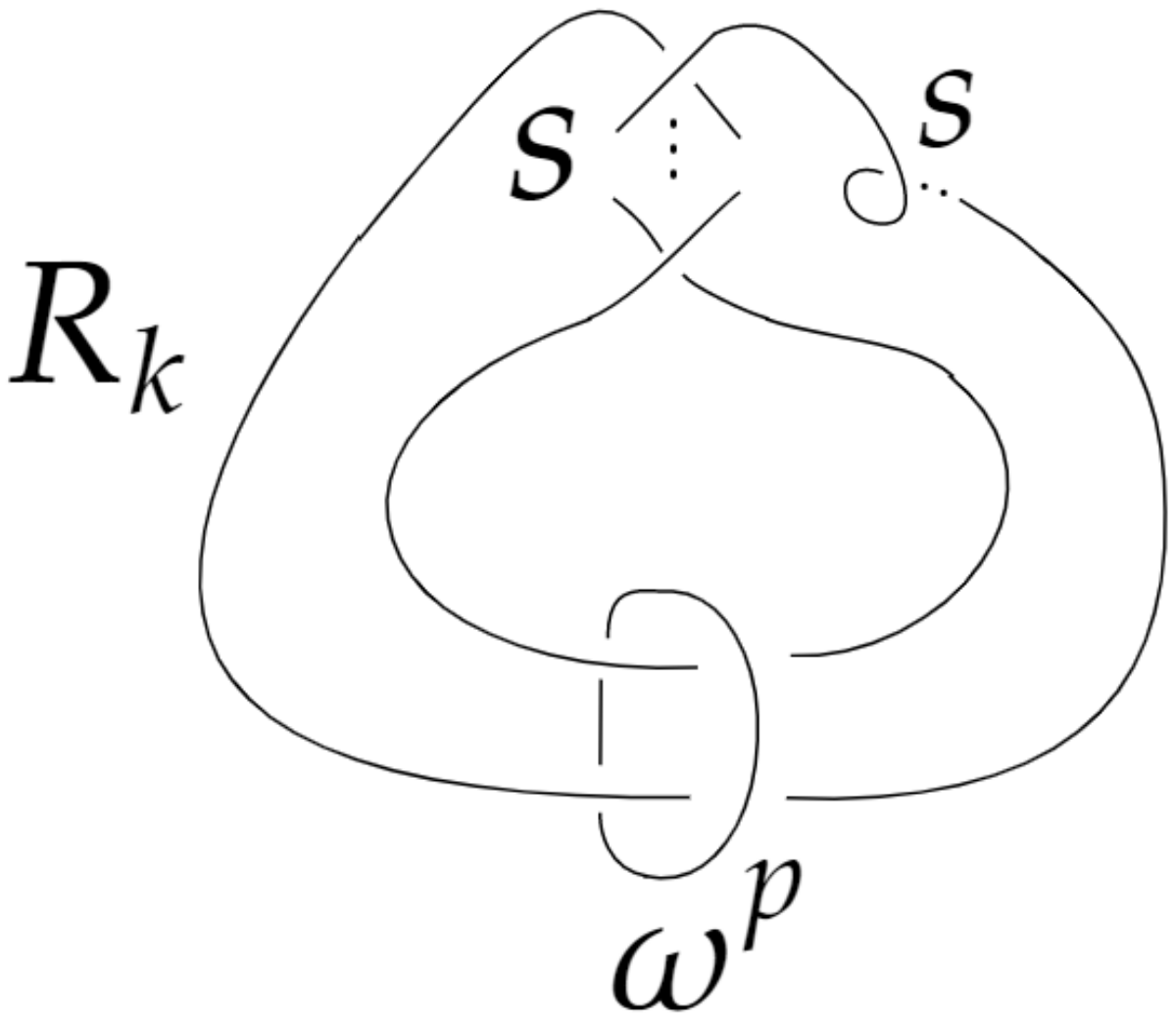}}\\ &=\sum_{k=0}^{N-1}
(-1)^{N-1-k}\left[
\begin{array}{@{\,}c@{\,}}N+k \\ N-1-k \end{array} \right]c_{k,p} \raisebox{-45pt}{\includegraphics[width=120
pt]{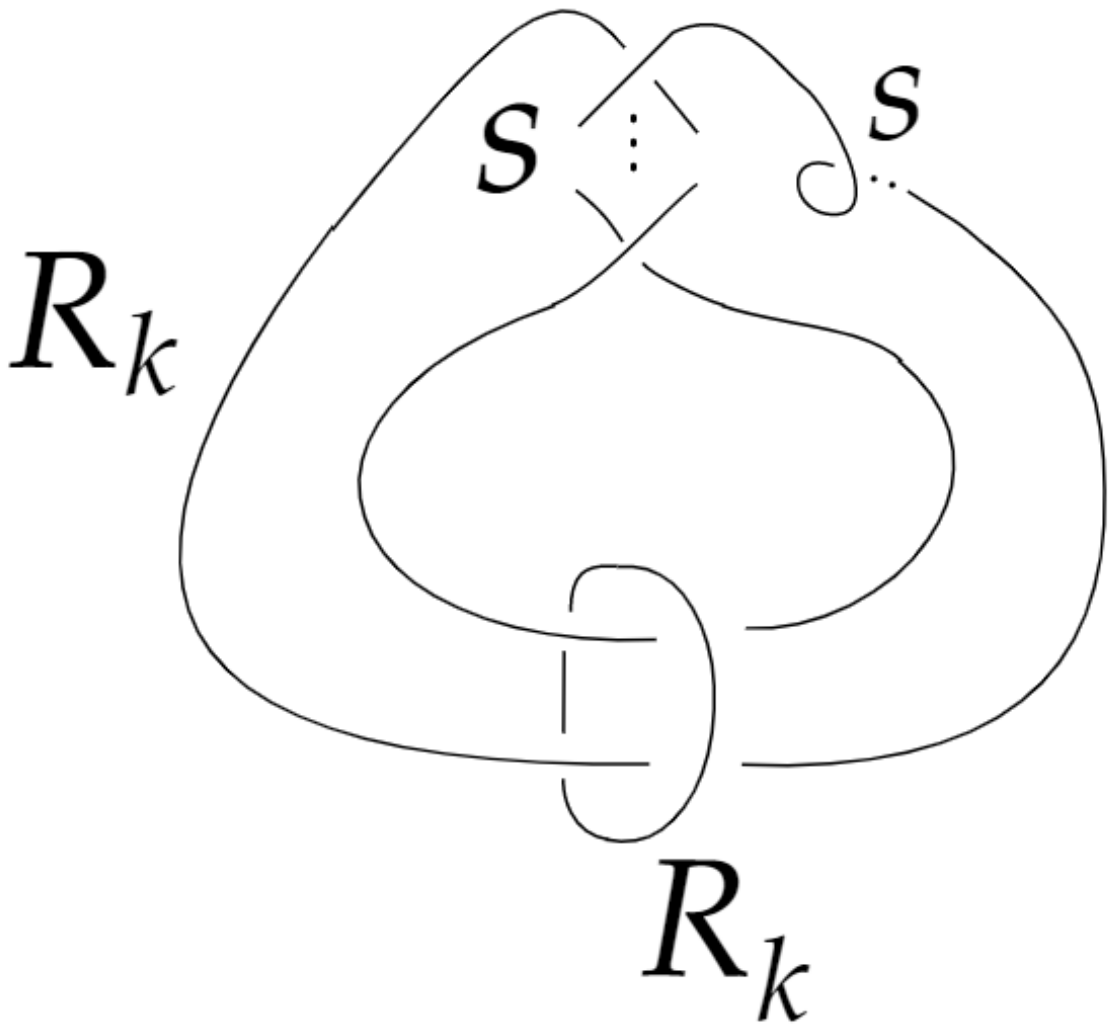}}\\ &=\sum_{k=0}^{N-1}
(-1)^{N-1-k}\left[
\begin{array}{@{\,}c@{\,}}N+k \\ N-1-k \end{array} \right]c_{k,p}\tilde{c}_{k,\frac{s}{2}}\frac{\{2k+1\}!(\{k\}!)^2}{\{1\}}\\ &=
\sum_{k=0}^{N-1}
(-1)^{N-1-k}c_{k,p}\tilde{c}_{k,\frac{s}{2}}(\{k\}!)^2\frac{\{N+k\}!}{\{N-1-k\}!\{1\}}. 
\end{align*}

Then
\begin{align*}
J_{N}(\mathcal{K}_{p,\frac{s}{2}};\mathfrak{q})&=
    \raisebox{-45pt}{
\includegraphics[width=120 pt]{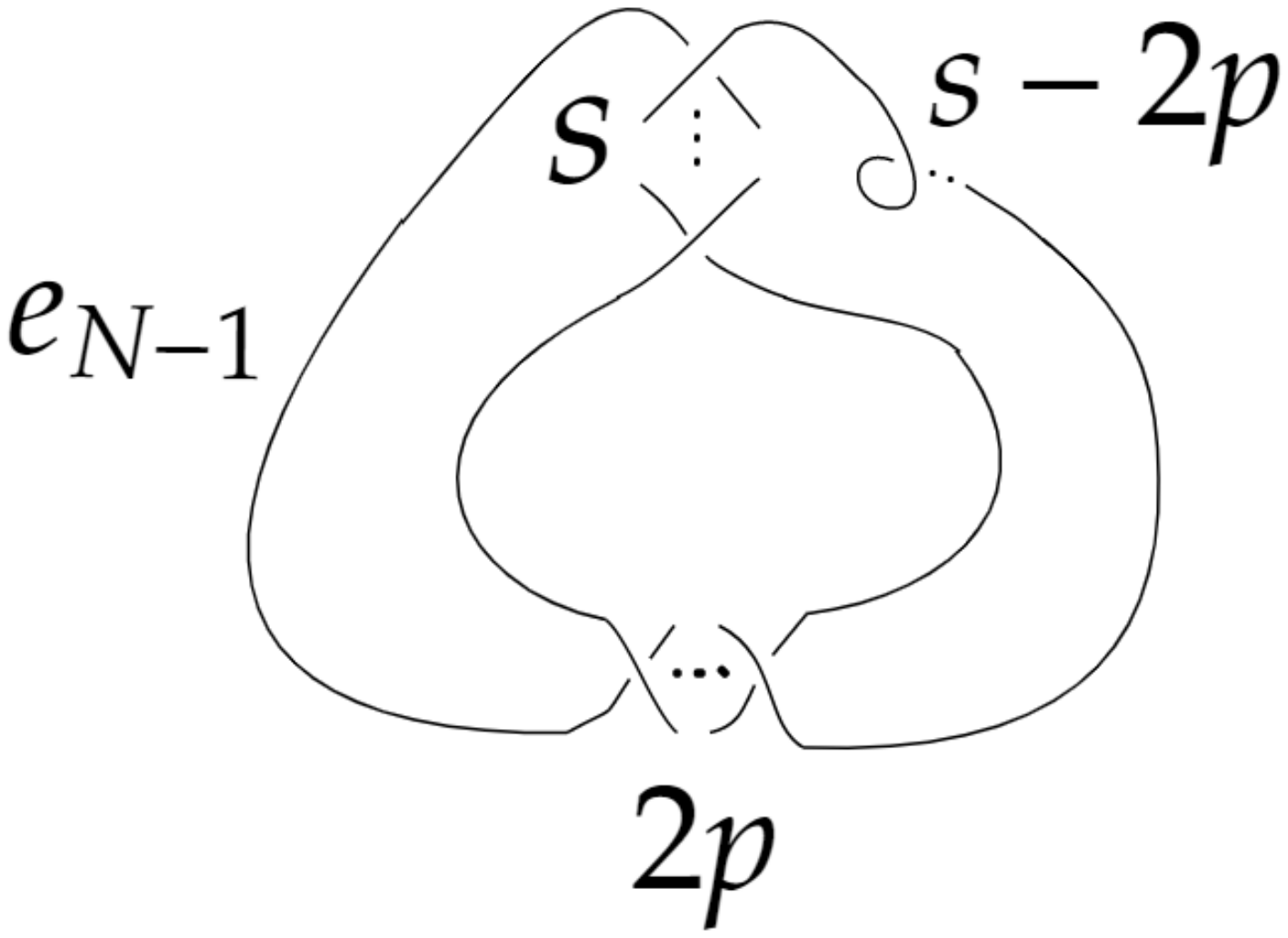}}\\
&=\raisebox{-45pt}{
\includegraphics[width=120 pt]{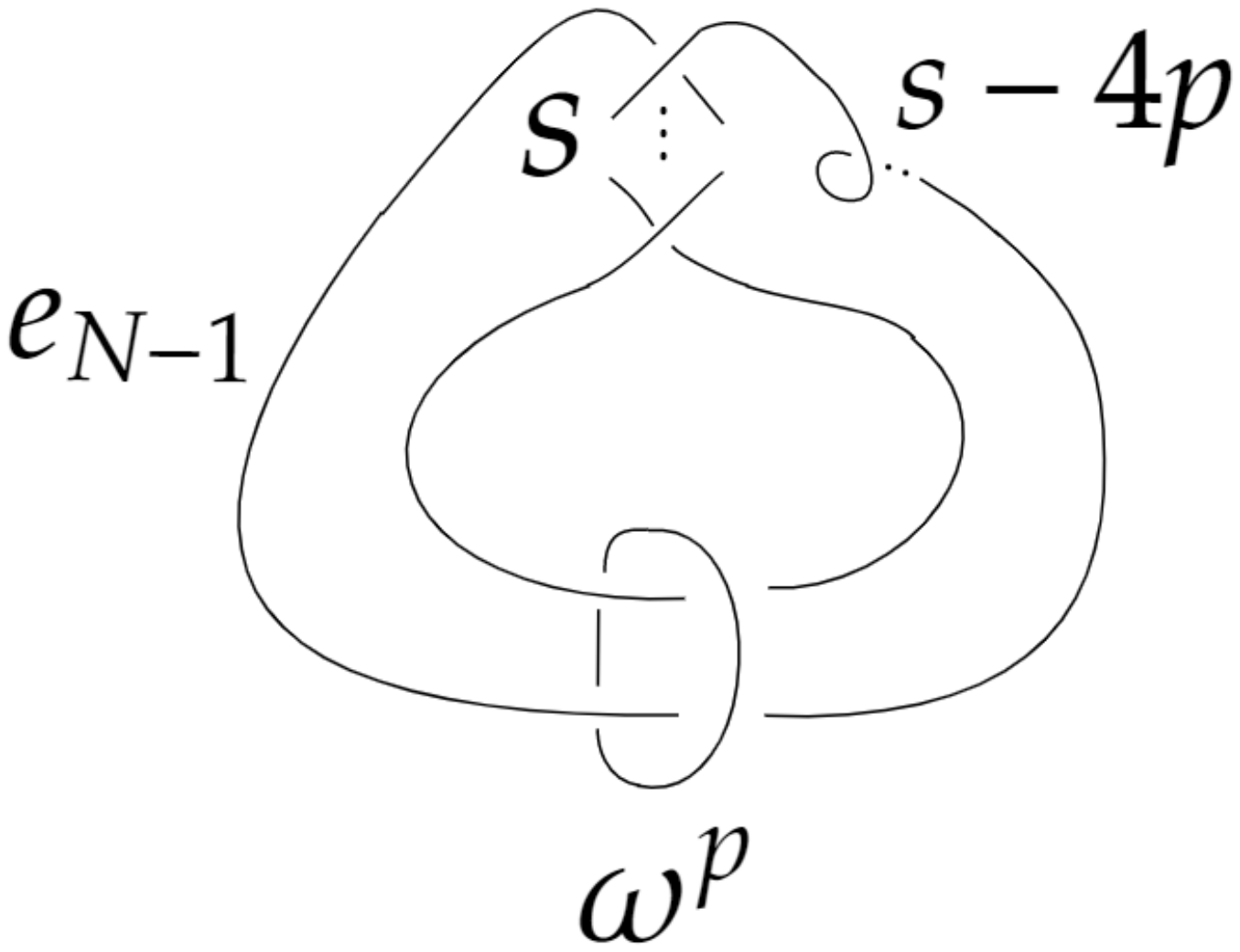}}\\
&=\mathfrak{q}^{-2p(N^2-1)}\raisebox{-45pt}{
\includegraphics[width=120 pt]{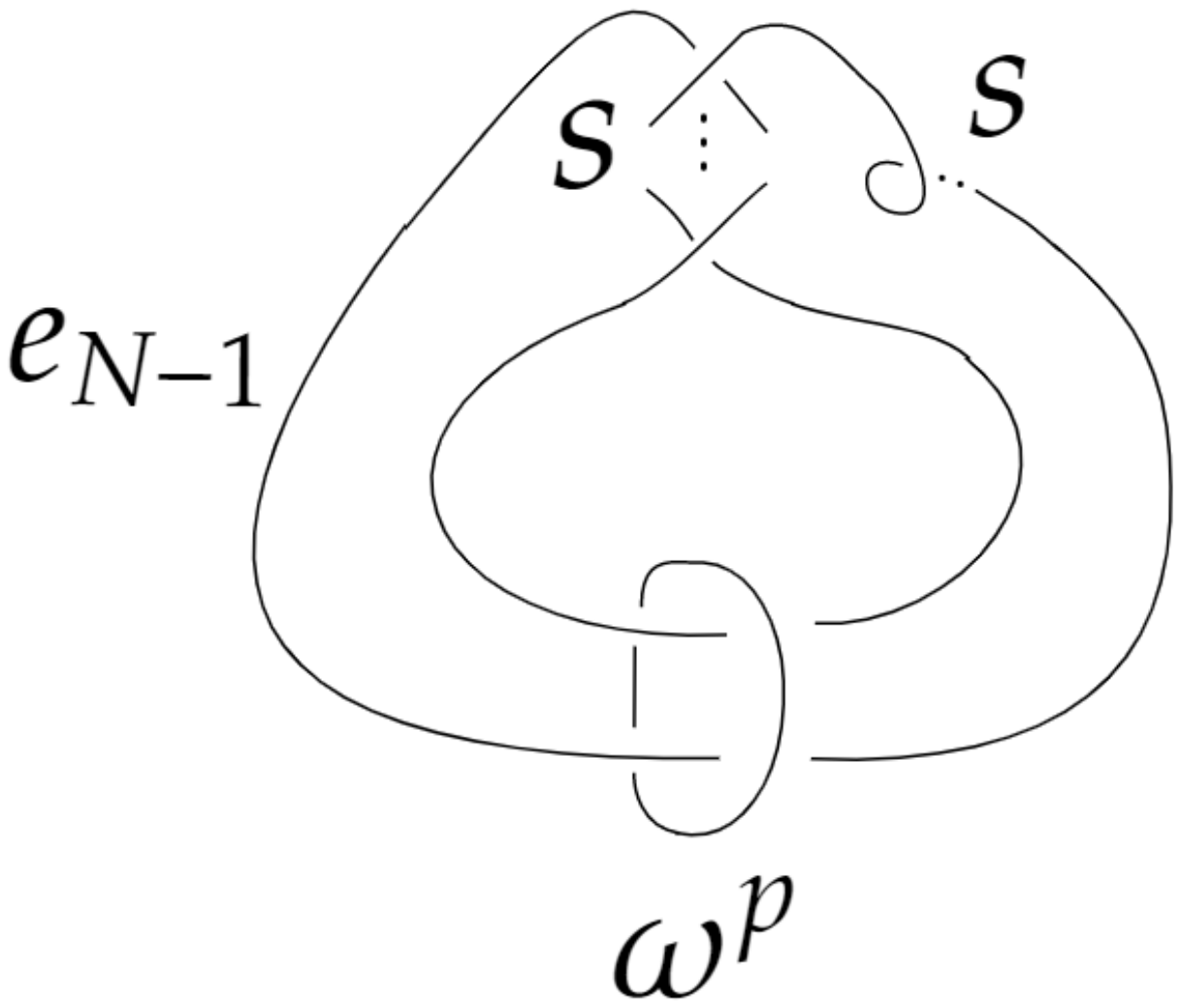}}\\
&=\mathfrak{q}^{-2p(N^2-1)}\sum_{k=0}^{N-1}
(-1)^{N-1-k}c_{k,p}\tilde{c}_{k,\frac{s}{2}}(\{k\}!)^2\frac{\{N+k\}!}{\{N-1-k\}!\{1\}}. 
\end{align*}

Hence
\begin{align*}
   J'_{N}(\mathcal{K}_{p,\frac{s}{2}};\mathfrak{q})&=\frac{(-1)^{N-1}\{1\}}{\{N\}}J_{N}(\mathcal{K}_{p,\frac{s}{2}};\mathfrak{q})\\
   &=\mathfrak{q}^{-2p(N^2-1)}\sum_{k=0}^{N-1}
(-1)^{k}c'_{k,p}\tilde{c}'_{k,\frac{s}{2}}\frac{\{N+k\}!}{\{N-1-k\}!\{N\}}.
\end{align*}
\end{proof}


\begin{thebibliography}{99}

\bibitem{And84} G. E. Andrews, {\em Multiple series Rogers-Ramanujan type identities}, Pacific Journal of
Mathematics 114 (1984) 267-283.

\bibitem{BHMV92} C. Blanchet, N. Habegger, G. Masbaum, P. Vogel, {\em Three-manifold invariants derived from Kauffman bracket}, Topology 31 (1992), 685–699.


\bibitem{CLZ21} Q. Chen, K. Liu and S. Zhu, {\em Cyclotomic expansions for the colored HOMFLY-PT invariants of double twist knots}, arXiv:2110.03616. 


\bibitem{Hab00} K. Habiro, {\em On the colored Jones polynomial of some simple links}. In: Recent
Progress Towards the Volume Conjecture, Research Institute for Mathematical
Sciences (RIMS) Kokyuroku 1172, September 2000.


\bibitem{Hab08} K. Habiro, \emph{A unified Witten-Reshetikhin-Turaev invariant
for integral homology spheres}, Invent. Math. 171 (2008), no. 1,
1-81.

\bibitem{Kaw21} K. Kawagoe, {\em The colored HOMFLY-PT polynomials of the trefoil knot, the figure-eight knot and twist knots}. arXiv:2107.08678.

\bibitem{Lau10} M. R. Lauridsen, {\em Aspects of quantum mathematics, Hitchin connections and AJ conjectures}, Ph.D. thesis,
Aarhus University, Aarhus, Denmark, 2010.

\bibitem{LO17} J. Lovejoy and R. Osburn, {\em  The colored Jones polynomial and Kontsevich-Zagier series for double twist knots}, 
arXiv:1710.04865.

\bibitem{LO19} J. Lovejoy and R. Osburn, {\em  The colored Jones polynomial and Kontsevich-Zagier series for double twist knots, II}, 
arXiv:1903.05060.



\bibitem{Mas03} G. Masbaum, {\em Skein-theoretical derivation of some formulas of Habiro}. \textit{Algebraic \& Geometric Topology}, \textbf{3}
(2003), 537-556.

\bibitem{MV94} G. Masbaum and P. Vogel, {\em 3-valent graphs and the Kauffman bracket}. Pacific J. Math, 164(2):361–381, 1994.

\bibitem{NRZ12} S. Nawata, P. Ramadevi, Zodinmawia and Xinyu Sun, {\em Super-A-polynomials for twist knots}, Journal of High Energy Physics, 2012. arXiv: 1209.1409.
 



\bibitem{Wal14} K. Walsh, {\em Patterns and stability in the coefficients of the colored Jones polynomial}, Ph.D. thesis, University
of California, San Diego, 2014.

\bibitem{War03} S.O. Warnaar, {\em Partial theta functions. I. Beyond the lost notebook}, Proc. London Math. Soc. (3) 87 (2003),
no. 2, 363-395.





\end{thebibliography}
\end{document}